\documentclass{article}

\usepackage[latin1]{inputenc}
\usepackage{amsmath}
\usepackage{amsfonts}
\usepackage{amssymb}
\usepackage{xcolor}
\usepackage{graphicx}
\usepackage{subcaption}
\usepackage{algorithmic}
\usepackage{algorithm}[0.1]

\usepackage{hyperref} 

\ifdefined\proof\else \newenvironment{proof}{\small{\bf Proof.}}{\hfill$\Box$\normalsize\bigskip} \fi

\newcounter{myenumerate}
\renewcommand{\themyenumerate}{\alph{myenumerate}}
\newenvironment{myenumerate}{\begin{list}{--(\themyenumerate)}{
\usecounter{myenumerate}\setlength{\labelsep}{0.5ex}
\setlength{\leftmargin}{0pt}\setlength{\labelwidth}{-\labelsep}
}}{\end{list}}

\ifdefined\theorem\else \newtheorem{theorem}{Theorem}\fi
\ifdefined\proposition\else \newtheorem{proposition}[theorem]{Proposition}\fi
\ifdefined\definition\else \newtheorem{definition}[theorem]{Definition}\fi
\ifdefined\lemma\else \newtheorem{lemma}[theorem]{Lemma}\fi
\ifdefined\corollary\else \newtheorem{corollary}[theorem]{Corollary}\fi
\ifdefined\remark\else \newtheorem{remark}[theorem]{Remark}\fi
\ifdefined\assumption\else \newtheorem{assumption}{Assumption}\fi
\ifdefined\example\else \newtheorem{example}{Example}\fi
\ifdefined\notations\else \newtheorem{notations}{Notations}\fi

\newcommand{\NomAlgo}{Tropical Dynamic Programming}
\newcommand{\nomalgo}{TDP}
\newcommand{\optop}{\mathop{\mathrm{opt}}}

\newcommand{\Id}{\;\mathrm{Id}}
\newcommand{\id}{\mathrm{I}}

\newcommand{\Pro}{\mathbb{P}}

\ifdefined\epi\else
\newcommand{\epi}{\mathop{\mathrm{epi}}}
\newcommand{\dom}{\mathop{\mathrm{dom}}}
\newcommand{\Sphere}{\mathbf{S}}
\newcommand{\argmin}{\mathop{\arg\min}}

\newcommand{\eqsepv}{\; , \enspace}
\newcommand{\eqfinv}{\; ,}
\newcommand{\eqfinp}{\; .}

\newcommand{\np}[1]{(#1)}
\newcommand{\bp}[1]{\big(#1\big)}
\newcommand{\Bp}[1]{\Big(#1\Big)}

\newcommand{\nc}[1]{[#1]}
\newcommand{\bc}[1]{\big[#1\big]}
\newcommand{\Bc}[1]{\Big[#1\Big]}

\newcommand{\na}[1]{\left\{#1\right\}}
\newcommand{\ba}[1]{\big\{#1\big\}}

\newcommand{\bset}[2]{\ba{#1\,\big|\, #2}}

\fi

\newcommand{\Mat}{\mathbb{M}}   
\newcommand{\Sy}{\mathbb{S}}    
\newcommand{\Sn}{\mathbb{S}^+} 
\newcommand{\Snp}{\mathbb{S}^{++}} 

\newcommand{\Oracle}{\text{Trial point Oracle}} 

\newcommand{\Selection}[2]{ 
  \def\tst{#1}
  \def\testt{#2}
  \ifx\testt\empty
    \ifx\tst\empty
      \mathcal{S}_t
    \else
      \mathcal{S}_t\left[ #1 \right]
    \fi
  \else
    \mathcal{S}_t\left[#1\right]\left(#2\right)
  \fi
}

\newcommand{\ce}[1]{[\![#1]\!]}         

\newcommand{\indi}[1]{\mathbb{\delta}_{#1}}    

\newcommand{\B}{\mathcal{B}} 

\newcommand{\discret}{v}	
\newcommand{\Discret}{\mathbb{V}}	

\newcommand{\func}{\phi} 
\newcommand{\Func}{F}
\newcommand{\Funcb}{\mathbf{F}} 
\newcommand{\Funcbb}{\mathbb{F}} 

\newcommand{\R}{\mathbb{R}} 
\newcommand{\N}{\mathbb{N}} 
\newcommand{\U}{\mathbb{U}} 
\newcommand{\X}{\mathbb{X}} 
\newcommand{\V}{\mathbb{V}} 
\newcommand{\Q}{\mathbb{Q}}
\newcommand{\ArbitrarySet}{\mathbb{E}}
\newcommand{\vmin}{\beta}
\newcommand{\vmax}{\gamma}
\newcommand{\vtilden}{\widetilde{V}_{t,N}}
\newcommand{\vopt}{\mathcal{V}}

\newcommand{\SDDP}{\text{\tiny SDDP}}
\newcommand{\MinPlus}{\text{\tiny min-plus}}

\newcommand{\lambdamin}{\lambda_{\tiny \text{min}}}
\newcommand{\lambdamax}{\lambda_{\tiny \text{max}}}

\usepackage{a4wide}
\title{A stochastic algorithm for deterministic multistage optimization problems}

\author{Marianne Akian\thanks{INRIA Saclay Ile-de-France and CMAP, \'{E}cole Polytechnique, CNRS, France.}
  \and Jean-Philippe Chancelier\thanks{CERMICS, \'{E}cole des Ponts ParisTech, France}
  \and Beno\^{i}t Tran\footnotemark[2]\,\,\footnotemark[1]
}

\begin{document}

\maketitle

\begin{abstract}
 Several attempts to dampen the curse of dimensionality problem of the Dynamic
 Programming approach for solving multistage optimization problems have been
 investigated. One popular way to address this issue is the Stochastic Dual
 Dynamic Programming method (SDDP) introduced by Perreira and Pinto in 1991 for
 Markov Decision Processes. Assuming that the value function is convex (for a
 minimization problem), one builds a non-decreasing sequence of lower (or
 outer) convex approximations of the value function. Those convex
 approximations are constructed as a supremum of affine cuts.

 On continuous time deterministic optimal control problems, assuming that the
 value function is semiconvex, Zheng Qu, inspired by the work of McEneaney,
 introduced in 2013 a stochastic max-plus scheme that builds upper (or inner)
 non-increasing approximations of the value function.

 In this note, we build a common framework for both the SDDP and a discrete
 time version of Zheng Qu's algorithm to solve deterministic multistage
 optimization problems. Our algorithm generates monotone approximations of the
 value functions as a pointwise supremum, or infimum, of basic (affine or
 quadratic for example) functions which are randomly selected.  We give
 sufficient conditions on the way basic functions are selected in order to
 ensure almost sure convergence of the approximations to the value function on
 a set of interest.
\end{abstract}

\section{Introduction}

Throughout this paper, we aim to study a deterministic optimal control problem
with discrete time. Informally, given a time $t$ and a state $x_t \in \X$, one
can apply a control $u_t \in \U$ and the next state is given by the dynamic
$f_t$, that is $x_{t+1} = f_t \left( x_t, u_t \right)$. Then, one wants to
minimize the sum of costs $c_t\left( x_t, u_t \right)$ induced by the controls
starting from a given state $x_0$ and during a given time horizon
$T$. Furthermore, one can add some final restrictions on the states at time $T$
which will be modeled by an additional cost function $\psi$ depending only on
the final state $x_T$.  We will call such optimal control problems,
\emph{multistage optimization problems} and \emph{switched multistage
 optimization problems} if the controls are both continuous and discrete:
\begin{subequations}
 \label{MultistageProblem}
 \begin{align}
  \min_{ \substack{x=(x_0, \ldots, x_T)\in \X^{T+1} \\ u = (u_0, \ldots u_{T-1})\in \U^T }}
   & \sum_{t=0}^{T-1} c_t \np{x_t, u_t} + \psi(x_T) \\
  \text{ s.t. }
   & \forall t \in \ce{0,T{-}1} \eqsepv             
  \ x_{t+1} = f_t\np{x_t, u_t} \text{ and } x_0 \in \X \ \text{ given}\eqfinp
 \end{align}
\end{subequations}

One can solve the multistage Problem~\eqref{MultistageProblem} by Dynamic
Programming as introduced by Richard Bellman around
1950~\cite{Be1954,Dr2002}. This method breaks the multistage
Problem~\eqref{MultistageProblem} into $T$ sub-problems that one can solve by
backward recursion over time. More precisely, denoting by
$\B_t : \overline{\R}^\X \to \overline{\R}^\X$ the operator from the set of
functions over $\X$ that may take infinite values to itself, defined by
\begin{equation}
 \label{BellmanOperator}
 \B_t\np{\func} : x\mapsto \min_{u\in \U} \Bp{c_t(x,u) + \func\bp{f_t(x,u)}}\eqfinv
\end{equation}
one can show (see for example~\cite{Be2016}) that solving Problem~\eqref{MultistageProblem} amounts to solve the following sequence of sub-problems:
\begin{equation}\label{DynamicProgramming}
 V_T  = \psi \quad \text{ and }\quad  \forall t\in \ce{0,T-1} \quad V_t = \B_t(V_{t+1})\eqfinp
\end{equation}
We will call each operator $\B_t$ the \emph{Bellman operator} at time $t$ and
each equation in~\eqref{DynamicProgramming} will be called the \emph{Bellman
 equation} at time $t$. Lastly, the function $V_t$ defined in
Equation~\eqref{DynamicProgramming} will be called the (Bellman) \emph{value
 function} at time $t$. Note that the value of Problem~\eqref{MultistageProblem} is
equal to the value function $V_0$ at point $x_0$, that is
$V_0 \left( x_0 \right)$, whereas solving the sequence of
sub-problems given by Equation~\eqref{DynamicProgramming} means to compute the value functions
$V_t$ at each point $x\in\X$ and time $t\in \ce{0,T{-}1}$.

We will state several assumptions on these operators in Section~\ref{sec:lAlgorithme}
under which we will devise an algorithm to solve the system of Bellman Equation~\eqref{DynamicProgramming},
also called the Dynamic Programming formulation of
the multistage problem. Let us stress on the fact that although we want to solve
the multistage Problem~\eqref{MultistageProblem}, we will mostly focus on its
(equivalent) Dynamic Programming formulation given by Equation~\eqref{DynamicProgramming}.

One issue of using Dynamic Programming to solve multistage optimization problems
is the so-called \emph{curse of dimensionality}~\cite{Be1954}. That is, when the
state space $\X$ is a vector space, grid-based methods to compute the value
functions have a complexity which is exponential in the dimension of the state
space $\X$. One popular algorithm
(see~\cite{Gi.Le.Ph2015,Gu2014,Gu.Ro2012,Pe.Pi1991,Sh2011,Zo.Ah.Su2018}) that
aims to dampen the curse of dimensionality is the Stochastic Dual Dynamic
Programming algorithm (or SDDP for short) introduced by Pereira and Pinto in
1991. Assuming that the cost functions $c_t$ are convex and the dynamics $f_t$
are linear, the value functions defined in the Dynamic Programming
formulation~\eqref{DynamicProgramming} are convex~\cite{Gi.Le.Ph2015}.  Under
these assumptions, the SDDP algorithm aims to build lower (or outer)
approximations of the value functions as suprema of affine functions and thus,
does not rely on a discretization of the state space. In the aforementioned
references, this approach is used to solve stochastic multistage convex
optimization problems, however in this article we will restrict our study to
deterministic multistage convex optimization problems as formulated in
Problem~\eqref{MultistageProblem}. Still, the SDDP algorithm can be applied to
our framework. One of the main drawback of the SDDP algorithm (in the stochastic
case) is the lack of an efficient stopping criterion: it builds lower
approximations of the value functions but upper (or inner) approximations are
built through a Monte-Carlo scheme that is costly and the associated stopping
criteria is not deterministic. We follow another path to provide upper
approximations as explained now.

In~\cite[Ch.\ 8]{Qu2013} and~\cite{Qu2014}, Qu devised an algorithm which builds
upper approximations of a Bellman value function arising in an infinite horizon
and continuous time framework where the set of controls is both discrete and
continuous. This work was inspired by the work of McEneaney~\cite{Mc2007} using
techniques coming from tropical algebra, also called max-plus or min-plus
techniques. Assume that $\X=\R^n$ and that for each fixed discrete control the
cost functions are convex quadratic and the dynamics are linear in both the
state and the continuous control. If the set of discrete controls is finite,
then exploiting the min-plus linearity of the Bellman operators $\B_t$, one can
show that the value functions can be computed as a finite pointwise infimum of
convex quadratic functions:
\begin{equation*}
 V_t = \inf_{\func_t \in \Func_t} \func_t\eqfinv
\end{equation*}
where $\Func_t$ is a finite set of convex quadratic forms. Moreover, in this
framework, the elements of $\Func_t$ can be explicitly computed through the
Discrete Algebraic Riccati Equation (DARE~\cite{La.Ro1995}). Thus, an
approximation scheme that computes an increasing sequence of subsets
$\left(\Func_t^k\right)_{k\in \N}$ of $\Func_t$ yields an algorithm that
converges after a finite number of improvements
\begin{equation*}
 V_t^k := \inf_{\func_t \in \Func_t^k} \func_t \approx \inf_{\func_t \in \Func_t} \func_t = V_t.
\end{equation*}
However, the size of the set of functions $\Func_t$ that need to be computed is
growing exponentially with $T-t$. In~\cite{Qu2014}, in order to address the
exponential growth of $\Func_t$, Qu introduced a probabilistic scheme that adds
to $\Func_t^k$ the ``best'' (given the current approximations) element of
$\Func_t$ at some point drawn on the unit sphere.

Our work aims to build a general algorithm that encompasses both a deterministic
version of the SDDP algorithm and an adaptation of Qu's work to a discrete time
and finite horizon framework.

The remainder of this paper is structured as follows. In
Section~\ref{sec:lAlgorithme}, we make several assumptions on the Bellman
operators $\B_t$ and define an algorithm which builds approximations of the
value functions as a pointwise optimum (\emph{i.e.} either a pointwise infimum
or a pointwise supremum) of basic functions in order to solve the associated
Dynamic Programming formulation~\eqref{DynamicProgramming} of the multistage
Problem~\eqref{MultistageProblem}.  At each iteration, the so-called basic
function that is added to the current approximation will have to satisfy two key
properties at a randomly drawn point, namely, \emph{tightness} and \emph{validity}. A key
feature of the proposed algorithm is that it can yield either upper or lower
approximations. More precisely,

\noindent $\bullet$ if the basic functions are affine, then approximating the
value functions by a pointwise supremum of affine functions will yield the SDDP
algorithm;

\noindent $\bullet$ if the basic functions are quadratic convex, then
approximating the value functions by a pointwise infimum of convex quadratic
functions will yield an adaptation of Qu's min-plus algorithm.

In Section~\ref{sec:Convergence}, we study the convergence of the approximations
of the value functions generated by our algorithm at a given time
$t\in \ce{0,T}$. We use an additional assumption on the random points on which
current approximations are improved, which state that they need to cover a ``rich enough set''
and show that the approximating sequence converges almost surely (over the
draws) to the Bellman value function on a set of interest.

In the last sections, we will specify our algorithm to three special cases. In
Section~\ref{SDDP_Example}, we prove that when building lower approximations as
a supremum of affine cuts, the condition on the draws is satisfied on the
optimal current trajectory, as done in SDDP. Thus, we get another point of view
on the usual (see~\cite{Gi.Le.Ph2015,Sh2011}) asymptotic convergence of SDDP,
in the deterministic case.  In Section~\ref{sec:Exemples_switch}, we describe an
algorithm which builds upper approximations as an infimum of quadratic forms. It
will be a step toward addressing the issue of computing efficient upper
approximations for the SDDP algorithm.  In Section~\ref{sec:ToyExample}, we
present on a toy example some numerical experiments where we simultaneously
compute lower approximations of the value functions by a deterministic version
of SDDP of the value functions and upper approximations of the value functions
by a discrete time version of Qu's min-plus algorithm.

\begin{figure}
 \centering
 \includegraphics[scale=0.7]{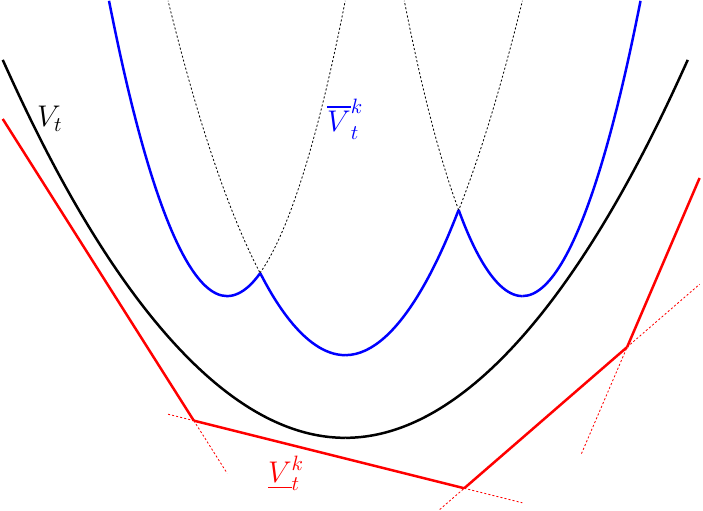}
 \caption{\label{fig:UpperLower}The lower approximations $\underline{V}_t^k$  will be built as a
  supremum of basic functions (here affine functions) that will always be
  below $V_t$. Likewise, the upper approximations $\overline{V}_t^k$
  will be built as an infimum of some other basic functions (here quadratic
  functions) that will always be above $V_t$.}
\end{figure}

\section{Notations and definitions}
\label{sec:lAlgorithme}

In the sequel, we will use the following notations

\noindent $\bullet$ $\X := \R^n$, endowed with its euclidean structure and its Borel $\sigma$-algebra
denotes the \emph{set of states}.

\noindent $\bullet$ $T$, a finite integer that we will call the time \emph{horizon}.

\noindent $\bullet$ $\optop$, denotes a generic operation that is either the \emph{pointwise infimum} or
the \emph{pointwise supremum} of functions which we will call the \emph{pointwise optimum}.

\noindent $\bullet$ $\overline{\R}$, denotes the extended real line endowed with the operations
$+\infty + \np{-\infty} = -\infty + \infty = +\infty$.

\noindent $\bullet$ $\dom{\phi}$, denotes the \emph{domain} of $\phi \in \left(\overline{\R}\right)^\X$ defined as
the subset of $\X$ in which $\phi(x)\in \R$.

\noindent $\bullet$ $\Funcb_t$ and $\Funcbb_t$, denote for every $t\in \ce{0,T}$, two subsets of the set $\left(\overline{\R}\right)^\X$
such that $\Funcb_t \subset \Funcbb_t$.

\noindent $\bullet$ $\func$ is said to be a \emph{basic function} if it is an element of $\Funcb_t$ for some $t\in \ce{0,T}$.

\noindent $\bullet$ $\indi{X}$ denotes,  for every set $X \subset \X$, the function equal to
$0$ on $X$ and $+ \infty$ elsewhere.

\noindent $\bullet$ For every $t\in \ce{0,T}$ and every set of basic functions $\Func_t \subset \Funcb_t$, we denote by $\vopt_{\Func_t}$ its pointwise optimum,
$\vopt_{\Func_t} := \optop_{\func \in \Func_t} \func$, that is
\begin{equation}
 \label{VoptF}
 \begin{array}{ccll}
  \vopt_{\Func_t} : & \X & \longrightarrow & \overline{\R}                                            \\
                    & x  & \longmapsto     & \optop \left\{\func (x) \mid \func \in \Func_t \right\}.
 \end{array}
\end{equation}

\noindent $\bullet$ $\left(\B_t\right)_{t\in \ce{0,T-1}}$ denotes a sequence of $T$ operators from $\overline{\R}^\X$ to $\overline{\R}^\X$,
called the \emph{Bellman operators}.

\noindent $\bullet$ $\left( V_t \right)_{t\in \ce{0,T}}$, denotes, for a fixed function $\psi : \X \to \overline{\R}$,
a sequence of \emph{value functions} given by the system of Bellman Equations~\eqref{DynamicProgramming}.

Now, we make several assumptions on the structure of
Problem~\eqref{DynamicProgramming}. These assumptions will be satisfied in the
examples developed in Sections~\ref{SDDP_Example} to~\ref{sec:ToyExample}.
These assumptions will make it possible to propagate
backward in time, regularity of the value function at the final time $t=T$ to
the value function at the initial time $t=0$.

\begin{assumption}[Structural assumptions]
 \label{Assumptions} \
 \begin{myenumerate}
  \item\label{Stability-pointwise-optimum} \textbf{Stability by pointwise optimum:}
  for every $t\in \ce{0,T}$, if $\Func_t \subset \Funcb_t$ then
  $\vopt_{\Func_t} \in \Funcbb_t$.
  \item\label{Stability-pointwise-convergence} \textbf{Stability by pointwise convergence:} for every $t\in \ce{0,T}$ if a sequence of functions $\np{\func^k}_{k\in \N}\subset \Funcbb_t$ converges pointwise to $\func$ on the domain of $V_t$, then $\func \in \Funcbb_t$.
  \item \label{CommonRegularity} \textbf{Common regularity:} for every $t\in \ce{0,T}$, there exists a common (local) modulus of continuity of all $\func\in\Funcbb_t$, \emph{i.e.} for every $x\in \dom\np{V_t}$, there exist $\omega_{t,x} : \R_+ \to \R_+ \cup \{+\infty \}$ which is increasing, equal to $0$ in $0$, continuous at $0$ and such that for every $\func \in \Funcbb_t$ and for every $x' \in \dom\np{V_t}$, we have that $\lvert \func(x) - \func (x') \rvert \leq \omega_{t,x} \np{\lVert x - x' \rVert}$.
  \item\label{Final-condition} \textbf{Final condition:} the value function $V_T$ at time $T$ is a pointwise optimum
  for some given subset $\Func_T$ of~$\Funcb_T$, that is $\psi := \vopt_{\Func_T}$.
  \item \label{StabilityBellman}
  \textbf{Stability by the Bellman operators:} for every $t\in \ce{0,T-1}$, if $\func \in \Funcbb_{t+1}$,
  then $\B_{t}\left( \func \right)$ belongs to $\Funcbb_{t}$.
  %
  %
  %
  \item\label{order-preserving} \textbf{Order preserving operators:} for every $t\in \ce{0,T-1}$, the operators $\B_t$ are \emph{order preserving}, \emph{i.e.} if $\phi,
   \varphi \in \Funcbb_{t+1}$ are such that $\phi \leq \varphi$, then $\B_t
   \left(\phi\right) \leq \B_t\left(\varphi\right)$.
  \item\label{Additively-subhomogeneous} \textbf{Additively subhomogeneous operators: } for every time step $t\in \ce{0,T-1}$, and every given compact set $K_t$, there exists $M_t> 0$ such that the operator $\B_t$ restricted to $K_t$ is \emph{additively subhomogeneous} over $\Funcbb_{t+1}$, meaning that for every constant function $\lambda\geq 0$ and  every function $\func \in \Funcbb_{t+1}$, we have
  \[
   \B_t\left( \func + \lambda \right) + \indi{K_t} \leq \B_t\left(\func\right) + \lambda M_t + \indi{K_t}.
  \]
  \item\label{proper-value} \textbf{Proper value functions:} the solution $\left(V_t\right)_{t\in \ce{0,T}}$ to the Bellman equations~\eqref{DynamicProgramming} never takes the value $-\infty$ and  is not identically equal to $+\infty$.
  \item\label{optimal-sets} \textbf{Compactness condition:} for every $t\in \ce{0,T-1}$ and every compact
  set $K_t \subset \dom \np{ V_t}$,
  there exists a compact set
  $K_{t+1} \subset \dom \np{ V_{t+1} }$ such that,
  for every function
  $\func \in \Funcbb_{t+1}$ and constant $\lambda \geq 0$,  we have
  \[
   \B_t\left( \func + \lambda + \indi{K_{t+1}} \right) \leq \B_t \left( \func + \lambda \right)
   + \indi{K_t}.
  \]
  %
 \end{myenumerate}
\end{assumption}

\begin{remark}
 Assumption~\ref{Assumptions}-\eqref{CommonRegularity} ensures that the domain of each function
 of $\Funcbb_t$ includes the domain of $V_t$. Note that if $\Funcbb_t$ is the
 set of all functions satisfying Assumption~\ref{Assumptions}-\eqref{CommonRegularity}, then
 Assumption~\ref{Assumptions}-\eqref{Stability-pointwise-convergence} is trivially
 satisfied. Also note that the domain of $V_t$ is known as in~\cite{Gi.Le.Ph2015}.
\end{remark}
\begin{remark}
 Note that Assumption~\ref{Assumptions}-\eqref{proper-value} and~\ref{Assumptions}-\eqref{optimal-sets} do not
 change whether $\optop = \inf$ or $\optop = \sup$ as the optimal control problem
 that we consider is formulated as a minimization problem.
\end{remark}

\begin{lemma}\label{ContinuiteVt}
 For every $t\in \ce{0,T}$ we have that $V_t \in \Funcbb_t$.
\end{lemma}

\begin{proof}
 By Assumption~\ref{Assumptions}-\eqref{Final-condition} and
 Assumption~\ref{Assumptions}-\eqref{Stability-pointwise-optimum}, $V_T$ is in
 $\Funcbb_T$. Now, assume that for some $t\in \ce{0,T-1}$ we have that
 $V_{t+1} \in \Funcbb_{t+1}$. By Assumption~\ref{Assumptions}-\eqref{StabilityBellman},
 we have that $V_t = \B_t\np{V_{t+1}} \in \Funcbb_t$ which ends the proof by
 backward induction.
\end{proof}

From a set of basic functions $\Func_t \subset \Funcb_t$, one can build its
pointwise optimum $\vopt_{\Func_t} = \optop_{\phi \in \Func_t} \phi$. We build a
monotone sequence of approximations of the value functions as optima of basic
functions which will be computed through \emph{compatible selection functions}
as defined below.  We illustrate this definition in Figure~\ref{fig:TightValid}.

\begin{definition}[Compatible selection function]
 \label{det:CompatibleSelection} Let a time step $t \in \ce{0,T-1}$ be fixed.
 A \emph{compatible selection function}, or simply \emph{selection function}, is a function $\Selection{}{}$ from
 $2^{\Funcb_{t+1}} \times \X$ to $\Funcb_t$ satisfying the two following properties

 \noindent -- \textbf{Validity}: for every set of basic functions $\Func_{t+1} \subset \Funcb_{t+1}$ and every $x\in \X$, we
 have $\Selection{\Func_{t+1}, x}{} \leq \B_t\left(\vopt_{\Func_{t+1}}\right)$
 (resp. $\Selection{\Func_{t+1}, x}{} \geq \B_t\left(\vopt_{\Func_{t+1}}\right)$) when
 $\optop = \sup$ (resp. $\optop = \inf$).

 \noindent -- \textbf{Tightness}: for every set of basic functions $\Func_{t+1} \subset \Funcb_{t+1}$ and every $x\in \X$
 the functions $\Selection{\Func_{t+1}, x}{}$ and $\B_t\left( \vopt_{\Func_{t+1}}
  \right)$ coincide at point $x$, that is
 \(
 \Selection{\Func_{t+1}, x}{x} = \B_t\left( \vopt_{\Func_{t+1}} \right)\np{x}.
 \)

 For $t= T$, we say that $\mathcal{S}_T : 2^{\Funcb_T} \times \X \to
  \Funcb_T$ is a \emph{compatible selection function} if
 it is \emph{valid} and \emph{tight}.
 There, $\mathcal{S}_T$ is \emph{valid} if, for every $\Func_T \subset \Funcb_T$ and $x\in \X$, the function
 $\mathcal{S}_T\left( \Func_T, x \right)$ remains below (resp. above) the value
 function at time $T$ when $\optop = \sup$ (resp.  $\optop = \inf$).  Moreover,
 the function $\mathcal{S}_T$ is \emph{tight} if it coincides with the value function at point $x$,
 that is for every $\Func _T\subset \Funcb_T$ and $x\in \X$, we have
 \( \mathcal{S}_T\left[ \Func_T, x \right]\left(x\right) = V_T(x).
 \)
\end{definition}

\begin{figure}
 \centering
 \includegraphics[scale=0.7]{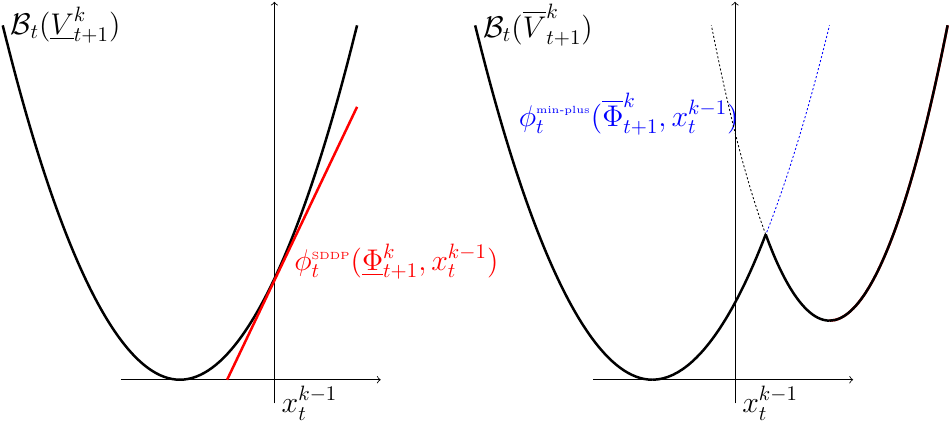}
 \caption{\label{fig:TightValid}In Sections~\ref{SDDP_Example} and~\ref{sec:Exemples_switch}, we will specify two selection
 functions, $\func_t^{\SDDP}$ and $\func_t^{\MinPlus}$, respectively, that
 will respectively yield upper and lower approximations of the value
 functions. In both cases, the selection function computes a basic function (in red or blue) which is
 equal, at the point $x_t^{k-1}$, to the image by the Bellman operator of the current
 approximation (in black), that is the tightness assumption. Moreover it remains
 above (or below) the image by the Bellman operator of the current approximation, that is the validity assumption.}
\end{figure}

Note that the Tightness assumption only asks for equality at the point $x$
between the functions $\func_t^\sharp\left(\Func_{t+1},x\right)$ and
$\B_t\left(\vopt_{\Func_{t+1}}\right)$ and not necessarily everywhere. The only
global property between the functions $\func_t^\sharp\left(\Func_{t+1},x\right)$
and $\B_t\left(\vopt_{\Func_{t+1}}\right)$ is an inequality given by the
validity assumption.

In Algorithm~\ref{\NomAlgo} we will generate, for every time $t$, a sequence of
random points of crucial importance that we will call \emph{trial points}. They
will be the ones where the selection functions will be evaluated, given the set
$\Func_t^k$ which characterizes the current approximation. In order to generate
those points, we will assume that we have at our disposition an Oracle which,
given $T{+}1$ sets of functions (characterizing the current approximations),
computes $T{+}1$ compact sets and a probability law.

\begin{definition}[Oracle]
 \label{AssumptionOracle}
 The Oracle takes as input $T{+}1$ sets of functions $F = \np{F_0,\ldots, F_T}$
 included in $\Funcb_0,\ldots, \Funcb_T$ respectively. Its output consists of
 $T{+}1$ compact sets $K_0, \ldots, K_T$, each included in $\X$, and of a
 probability measure ${\Pro}_F$ on the space
 $\X^{T+1}$ 
 which are such that

 \noindent -- Initialization. For every $t\in \ce{0,T}$, set $\Func_t = \emptyset$ and return $T{+}1$ given compact sets and a given probability measure.

 \noindent -- For every $t\in \ce{0,T}$, $K_t \subset \dom\left(V_t\right)$.

 \noindent -- The support of ${\Pro}_F$ is $K_0 \times \ldots \times K_T$.
\end{definition}

For every time $t\in \ce{0,T}$, we construct a sequence of functions
$\left(  V_t^{k} \right)_{k\in \N}$ belonging to $\Funcbb_t$ as
follows.  For every time $t\in \ce{0,T}$ and for every $k \geq 0$, we build a subset $\Func_t^{k}$ of the set $\Funcb_t$ and define the sequence of functions by pointwise optimum
\begin{equation}
 \label{Vktdef}
 V_t^{k} := \vopt_{\Func_t^k} = \optop_{\func  \in \Func_t^{k}} \func\eqfinp
\end{equation}
As described here, the functions are just byproducts of Algorithm~\ref{\NomAlgo}, which only
describes the way the sets $\Func_t^{k}$ are computed.

As the following algorithm was inspired by Qu's work which uses tropical algebra techniques, we will call this algorithm ``Tropical Dynamic Programming''.

\begin{algorithm}
 \caption{\NomAlgo \ (\nomalgo)}
 \label{\NomAlgo}
 \begin{algorithmic} \REQUIRE{For every $t\in \ce{0,T}$, $\Selection{}{}$ a compatible selection function and a $\Oracle$ satisfying Definition~\ref{AssumptionOracle}.}
  \ENSURE{For every $t\in \ce{0,T}$, a sequence of sets
   $\left(\Func_t^k\right)_{k\in\N}$ and the associated sequence $V_t^k = \optop_{\func \in \Func_t^k} \func$.}
  \STATE{Define for every $t\in \ce{0,T}$,
   $\Func_t^0 := \emptyset$.}
  \FOR{$k\geq 0$}
  \STATE{\emph{Forward phase}}
  \STATE{Compute $\left(K_0^k, \dots K_T^k, \Pro^{k}\right) = \text{Oracle}\left(\Func^k_0, \dots, \Func^k_T \right).$}
  \STATE{Draw trial points $\left(x_t^{k}\right)_{t\in \ce{0,T}}$ over $K_0^{k} \times K_1^{k} \times \ldots \times K_T^{k}$ according to $\Pro^{k}$ knowing the past iterations.}
  \STATE{\emph{Backward phase}}
  \STATE{Compute $\func_T^{k+1} := \mathcal{S}_T \left[\Func^{k}_{T}, x_T^{k}\right]$.}
  \STATE{Define $\Func_T^{k+1} := \Func_T^{k} \cup \left\{ \func_T^{k+1} \right\}$.}
  \FOR{$t$ from $T-1$ to $0$}
  \STATE{Compute $\func_t^{k+1} := \Selection{\Func^{k+1}_{t+1}, x_t^{k}}{}$.}
  \STATE{Define $\Func_t^{k+1} := \Func_t^{k} \cup \left\{ \func_t^{k+1} \right\}$.}
  \ENDFOR
  \ENDFOR
 \end{algorithmic}
\end{algorithm}

At each iteration, Algorithm~\ref{\NomAlgo} generates a trial point $x_t^k$ which only depends on the data available at the current iteration. We loosely explain this point. Define for every $k\in \N$,
$F^k = \np{F_t^k}_{t\in \ce{0,T}}$ and $x^k = \np{x_t^k}_{t\in \ce{0,T}}$. Then, there exists a deterministic function $\xi$ and a sequence of independent random variables $\np{W^k}_{k\in \N}$ such that for every $k\in \N$, $x^k = \xi\np{F^k, W^{k}}$ where $\np{W^{k}}_{k \in \N}$ is furthermore independent from $F^0$. Throughout the remainder of the article, denote by $\np{\Omega,\mathcal{F}, \Pro}$ a probability space on which the random variables $\np{W^k}_{k\in \N}$ are defined and independent.


We will denote by $+$ the Minkowski sum between sets, by $\mathbb{B}$ the unit
closed euclidean ball of $\X^{T+1}$ and for every $x \in \X^{T+1}$ and radius
$r>0$, $B(x,r)$ is the euclidean open ball of radius $r$ centered at
$x$. Furthermore, we define for every $t$, $K_t^* := \limsup_{k\in \N} K_t^k$
the set of all possible limit points of $K_t^k$.  We make the following
assumption on the Oracle which, loosely stated, ensures that if a state $x_t$ is
close to $K_t^*$, then $x_t$ is almost a limit point of the sequence of trial
points $\np{x_t^k}_{k\in \N}$.

\begin{assumption}[Trial point assumption]
 \label{TrialpointAssumption}
 For every radius $r' > 0$, there exists $r >0$ such that
 \begin{equation}
  \label{eq:TrialpointAssumption}
  \forall x \in \X
  \eqsepv
  \Pro
  \Bc{\bp{x \in \limsup_{k\in \N}\footnotemark K^k  + r\mathbb{B}} \Rightarrow x \in \limsup_{k\in \N} B(x^k,r') } = 1
  \eqfinp
 \end{equation}
\end{assumption}
\footnotetext{See~\cite[Definition 4.1 p. 109]{Ro.We2009}.}
Remark that $\np{\limsup_{k\in \N} K^k} + r\mathbb{B} =  \limsup_{k\in \N} \np{K^k + r\mathbb{B}}$, hence the lack of parenthesis. The following lemma gives some more insight on the Trial point assumption.

\begin{lemma}
 \label{lemmatrial}
 Consider the sequence of trial points $\np{x^k}_{k\in \N}$ generated by Algorithm~\ref{\NomAlgo} with an Oracle satisfying
 Assumption~\ref{TrialpointAssumption}. Given $r' > 0$ and $x\in \X$, for every $r'' > r'$, $\Pro$-a.s.,
 \begin{equation}
  \label{eq:lemmatrial1}
  x \in \limsup_{k\in \N} B(x^k, r') \Rightarrow x^k \in B(x,r'') \ \text{for infinitely many indices} \ k\in \N
  \eqfinp
 \end{equation}
 Conversely, given $r'' >0$ and $x\in \X$, for every $r' > r''$, $\Pro$-a.s.
 \begin{equation}
  \label{eq:lemmatrial2}
  x^k \in B(x,r'') \ \text{for infinitely many indices} \ k\in \N \Rightarrow x \in \limsup_{k\in \N} B(x^k, r')
  \eqfinp
 \end{equation}
\end{lemma}

\begin{proof}
 First, we prove Equation~\eqref{eq:lemmatrial1}. Fix $r''>r'>0$ and assume that $x \in \limsup_{k\in \N} B(x^k, r')$, $\Pro$-a.s..
 Then, there  exists an increasing function $\sigma : \N \to \N$ and a sequence
 $\np{y^{\sigma(k)}}_{k\in \N} \subset r'\mathbb{B}$ such that
 $x^{\sigma(k)} + y^{\sigma(k)} \underset{k\to +\infty}{\longrightarrow} x$. As
 $r'' - r' > 0$, there exists a rank $k_0\in \N$ such that when $k \geq k_0$ we have
 \(
 \lVert x - x^{\sigma(k)} + y^{\sigma(k)} \rVert \leq r'' - r'
 \). By triangle inequality, we have
 \[
  \lVert x - x^{\sigma(k)}\rVert \leq \lVert x - x^{\sigma(k)} + y^{\sigma(k)}\rVert + \lVert -y^{\sigma(k)} \rVert\\
  \leq (r''-r') + r' = r'',
 \] \emph{i.e.} $\Pro$-a.s., for every $k \geq k_0$, $x \in B(x^k, r'')$, which yields
 Equation~\eqref{eq:lemmatrial1}.

 Second, we prove Equation~\eqref{eq:lemmatrial2}. Fix $r'> r''> 0$ and assume that
 $x^k \in B(x,r'')$ for infinitely many indices $k\in \N$. Thus,
 $\Pro$-a.s, there exists an increasing function $\sigma : \N \to \N$ and a
 sequence $\np{y^{\sigma(k)}}_{k\in \N} \subset r''\mathbb{B}$ such that
 $x^{\sigma(k)} - x = y^{\sigma(k)}$. As $r' > r''$, $\Pro$-a.s.
 $x \in B(x^k, r')$ and
 \(
 x = x^{\sigma(k)} - y^{\sigma(k)}.
 \)
 Hence, we obtain Equation~\eqref{eq:lemmatrial2}.
\end{proof}

Now, we give two examples of Oracles that satisfy the Trial point assumption~\ref{TrialpointAssumption}. They are used respectively in Section~\ref{SDDP_Example} and~\ref{sec:Exemples_switch}.

\begin{example}[Independant uniform draws over the unit sphere]
 \label{example:sphere}
 Consider the Oracle which constantly outputs $T+1$ times the unit euclidean
 sphere $\Sphere$ of $\X$ and the uniform probability measure
 $\Pro^k := \sigma_U$ of $\Sphere^{T+1}$ on $\X^{T+1}$\footnote{For every
  $A \in \mathcal{B}\np{\X^{T+1}}$,
  $\sigma_U(A) = C \, \mathrm{Leb}\np{\pi_{\Sphere^{T+1}}^{-1}\np{A \cap
     \Sphere^{T+1}}}$, where $\mathrm{Leb}$ is the Lebesgue measure on
  $\X^{T+1}$, $\pi_{\Sphere^{T+1}}$ is the euclidean projector on $\Sphere^{T+1}$
  restricted to the ball $B\np{0,1}^{T+1}$ without $0$ and $C$ a normalization constant.}. Here, we have for every
 $k\in \N$, $K^k = \Sphere^{T+1}$. Fix an arbitrary $r' > 0$
 and set $r= r'/2$, we prove that
 \begin{equation*}
  \forall x \in \X, {\Pro}\Bc{   x \in \np{ \Sphere^{T+1}  + r\mathbb{B}} \Rightarrow x \in \limsup_{k\in \N} B(x^k,r') } = 1 \; .
 \end{equation*}
 \begin{proof}
  Fix $x\in \Sphere^{T+1} + r\mathbb{B}$, we need to show that we have
  \(
  \Pro\bc{x \in \limsup_{k\in \N} B(x^k,r') } = 1
  \). Now, fix $r''>0$ such that $r < r'' < r'$. Using Lemma~\ref{lemmatrial}-\eqref{eq:lemmatrial2}, it is enough to show that
  \begin{equation}
   \label{eq:oracle1}
   \Pro\nc{ x^k \in B(x,r'') \ \text{for infinitely many indices} \ k\in \N} = 1.
  \end{equation}
  As $r'' > r$ and $x$ is distant from $\Sphere^{T+1}$ by less than $r$, the
  quantity
  $\Pro\nc{x^k \in B(x,r'')} = \sigma_U \nc{B(x,r'') \cap \Sphere^{T+1}}$ is a
  positive constant in $k$. Thus, we have that
  \(
  \sum_{k\in \N} \Pro\bc{x^k \in B(x,r'')} = +\infty
  \).
  Moreover, the sequence of events $\np{x^k \in B(x,r)}_{k\in \N}$ are independent,
  thus by Borel-Cantelli's Lemma, Equation~\eqref{eq:oracle1} holds.
 \end{proof}
\end{example}

\begin{example}[Dirac on the current optimal trajectory]
 \label{example:opt_traj}
 The sequence of Probability measures $\np{{\Pro}^k}_{k\in \N}$ is recursively build as follows:

 \noindent -- Set ${\Pro}^0 := \np{\delta_{x_0^0}, \ldots \delta_{x_T^0}}$ where $x_t^0 \in K_t^0$ for every $t\in \ce{0,T}$.

 \noindent -- Given sets of functions $F_0^k, \ldots, F_T^k$. Start, by fixing $x_0^{k} = x_0$ and
 compute forward in time, for $t\in \ce{0,T{-}1}$, optimal controls by
 \(
 u_t^k \in \argmin_u \B_t^u\np{\vopt_{F_{t+1}^k}}(x_t^k),
 \)
 and successive states by
 \(
 x_{t+1}^k = f_t\np{x_t^k, u_t^k}
 \).

 \noindent -- Define a probability measures ${\Pro}^k := \np{\delta_{x^k_0}, \ldots, \delta_{x^k_T}}$.

 Consider the Oracle which, given sets of functions $F_0^k, \ldots, F_T^k$,
 outputs the singleton $\na{x^k} = \ba{\np{x^k_t}_{t\in \ce{0,T}}}$ and
 the probability measure
 ${\Pro}^k := \np{\delta_{x^k}}$  defined at previous step. Fix $r > 0$, take $r' = r >0$ and
 $x \in \X^{T+1}$. We obtain that
 \begin{equation*}
  {\Pro}\Bc{   x \in \bp{\limsup_{k\in \N}  \underbrace{\{x^k\}  + r\mathbb{B}}_{= B(x^k, r)} }^c \ \text{or} \ x \in \limsup_{k\in \N} B(x^k,r) } = 1 \; ,
 \end{equation*}
 which is equivalent to the Trial point assumption with $K^k = \{x^k\}$.
\end{example}

\section{Almost sure convergence on the set of accumulation points}

\label{sec:Convergence}

In this section, we will prove the convergence result stated in
Theorem~\ref{ConvergenceTheorem}. For this purpose, we state several crucial properties
of the approximation functions $\left(V_t^k\right)_{k\in \N}$ generated by
Algorithm~\ref{\NomAlgo}. They are direct consequences of the facts that the Bellman
operators are order preserving and that the basic functions building our
approximations are computed through compatible selection
functions. The Algorithm~\ref{\NomAlgo} is stochastic, as trial points are drawn at each
iteration from $\Pro^k$. Therefore, equalities, inequalities and statements
where the functions $V_t^k$ are involved hold $\Pro$-almost surely. However, for
the sake of simplicity, we will refrain from always adding $\Pro$-almost surely
in equalities, inequalities and some statements.

\begin{lemma}
 \label{ProprietesFaciles}
 The sequence of functions $\left(  V_t^{k} \right)_{k\in \N}$,
 for every $t\in \ce{0,T}$, given by Equation~\eqref{Vktdef} and
 produced by Algorithm~\ref{\NomAlgo} satisfy the following properties.
 \begin{enumerate}
  \item\label{proofitema} \textbf{Monotone approximations:} for every indices
        $k< k'$ and every $t\in \ce{0,T}$, we have that
        $V_t^k \geq V_t^{k'} \geq V_t$
        when $\optop = \inf$ and $V_t^k \leq V_t^{k'} \leq V_t$ when $\optop = \sup$.
  \item\label{proofitemb} For every $k\in \N$ and every $t\in \ce{0,T-1}$, we have that
        $
         \B_t\left( V_{t+1}^k\right) \leq V_t^k$ when $\optop = \inf$
        and $\B_t\left( V_{t+1}^k\right) \geq
         V_t^k$ when $\optop = \sup$.
  \item\label{proofitemc}
        For every $k\geq 1$ and every $t\in \ce{0,T-1}$, we have
        \(
        \B_t\left( V_{t+1}^k \right)\left( x_t^{k-1} \right) =
        V_t^k \left(x_t^{k-1} \right)
        \).
  \item\label{proofitemd} For every $k\geq 1$, we have
        \(   V_{T}^k\left( x_T^{k-1} \right) = V_T \left(x_T^{k-1} \right) \).
 \end{enumerate}
\end{lemma}

\begin{proof} We prove each point when $\optop = \inf$. The case $\optop = \sup$ is similar and left to the reader.

 \noindent $\bullet$~\eqref{proofitema} (left inequality).
 Let $t\in \ce{0,T}$ be fixed. 
 By construction of Algorithm~\ref{\NomAlgo}, the sequence of sets
 $\bp{\Func_t^k}_{k\in \N}$ is non-decreasing. Now, using the definition of the sequence $\left(V_t^k\right)_{k\in \N}$
 given by Equation~\eqref{Vktdef} we have that
 \( V_t^{k+1} = \vopt_{\Func^{k+1}_t}(x) = \inf_{\func\in \Func^{k+1}_t}\func(x) \leq \inf_{\func
  \in \Func^{k}_t} \func(x) = \vopt_{\Func^{k}_t}(x) = V_t^k
 \) and thus the sequence $\left(V_t^k \right)_{k\in \N}$ is non-increasing.

 \noindent $\bullet$~\eqref{proofitemb}. We prove the assertion by induction on
 $k\in \N$.  For $k=0$, as $\Func_t^0 = \emptyset$, we have $V_t^0 = +\infty$
 for all $t\in \ce{0,T-1}$ and thus the assertion is true.  Now, assume that for
 some $k\in \N$, we have for all $t\in \ce{0,T-1}$
 \begin{equation}
  \label{eq:HR}
  \B_t\left( V_{t+1}^{k} \right) \leq V_t^k \eqfinp
 \end{equation}
 Since $\bp{V_{t+1}^{k'}}_{k'\in \N}$ is non-increasing by already proved Item~\eqref{proofitema} and
 $\B_t$ is order preserving using Assumption~\ref{Assumptions}-\eqref{order-preserving}, we have that
 $\B_t\left( V_{t+1}^{k+1} \right)\leq \B_t\np{V_{t+1}^{k}}$. This last inequality combined with induction assumption
 given by Equation~\eqref{eq:HR} gives the inequality
 \begin{equation}
  \label{eq:HRbis} \B_t\left( V_{t+1}^{k+1} \right) \leq V_t^k\eqfinp
 \end{equation}
 Moreover, we also have that
 \begin{align}
  \B_t\left( V_{t+1}^{k+1} \right)
   & \mathop{=}_{(\text{by }~\eqref{Vktdef})} \B_t\left( \vopt_{F_{t+1}^{k+1}}\right)
  \mathop{\le}_{(\text{by } \mathcal{S}_t \text{ validity at } x_t^k)}
  \Selection{\Func_{t+1}^{k+1}, x^k_t}{}
  = \func^{k+1}_t \eqfinv
  \label{eq:HRter}
 \end{align}
 where the last equality is obtained by definition of function $\func^{k+1}_t$ in Algorithm~\ref{\NomAlgo}.
 Thus, combining Equation~\eqref{eq:HRbis} and~\eqref{eq:HRter} we have that
 $ \B_t\np{V_{t+1}^{k+1}} \le \inf \bp{V_t^{k} , \func_t^{k+1}}$.
 Finally, using Equation~\eqref{Vktdef} and Algorithm~\ref{\NomAlgo}, we have that
 \begin{equation*}
  \inf \bp{V_t^{k} , \func_t^{k+1}} =
  \inf \Bp{\inf_{\func \in \Func_t^{k} } \func, \func_t^{k+1}} =
  \inf_{\func \in \Func_t^{k} \cup \{ \func_t^{k+1}\}} \func =
  \inf_{\func \in \Func_t^{k+1}} \func =
  V_t^{k+1}\eqfinp
 \end{equation*}
 Thus, we obtain that
 \(
 \B_t\left( V_{t+1}^{k+1} \right) \leq \inf \bp{V_t^{k} , \func_t^{k+1}}=  V_t^{k+1}
 \), which gives the induction assumption for $k+1$ and concludes the proof of~\eqref{proofitemb}.

 \noindent $\bullet$~\eqref{proofitemc}. As the selection function $\Selection{}{}$ is tight in the sense of
 Definition~\ref{det:CompatibleSelection}, we have by construction of Algorithm~\ref{\NomAlgo} that
 $\B_t \np{V_{t+1}^k}\np{x_t^{k-1}} = \func_t^k \np{x_t^{k-1}}$.
 Combining this equation with Item~\eqref{proofitemb} 
 and the definition of $V_t^k$, one gets Lemma~\ref{ProprietesFaciles}-\eqref{proofitemc}. 

 \noindent $\bullet$~\eqref{proofitemd}. Similarly,
 we have that
 $V_{T} \np{ x_T^{k-1}}= \func_T^k \np{x_T^{k-1}}$,
 which combined with the inequality given in Item~\eqref{proofitema} 
 and the definition of $V_T^k$ gives Lemma~\ref{ProprietesFaciles}-\eqref{proofitemd}. 

 \noindent $\bullet$~\eqref{proofitema} (right inequality).
 We prove that $V_t^k \geq V_t$ for all for all $k\in \N$ and all $t\in \ce{0,T}$.
 Fix $k \in \N$, we show that $V_t^k \geq V_t$ for all $t\in \ce{0,T}$ by backward recursion on time $t$.
 For $t=T$, by validity of the selection functions
 given in Definition~\ref{det:CompatibleSelection}, for every $\func \in \Func_T^k$, we have that $\func \geq
  V_T$. Thus $V_T^k = \vopt_{\Func_T^k} = \inf_{\func \in \Func_T^k} \func \geq V_T$.  Now, suppose
 that for some $t \in \ce{0,T-1}$, we have that $V_{t+1} \le V_{t+1}^k$. Then, using the definition of the value function
 in Equation~\eqref{DynamicProgramming}, the fact that the Bellman operators are order preserving and the inequality already proved
 in Item~\eqref{proofitemb} we obtain that:
 \(
 V_t = \B_t \bp{V_{t+1}} \le \B_t \bp{V_{t+1}^k} \le V_t^{k}\eqfinv
 \)
 which gives the assertion for time $t$.
 \medskip
 This ends the proof.
\end{proof}

In the following two propositions, we state that the sequences
$\left( V_t^k \right)_{k\in \N}$ and
$\np{\B_t \left( V_{t+1}^k \right)}_{k\in \N}$ converge uniformly on any compact
included in the domain of $V_t$. The limit function $V_t^*$ of
$\left( V_t^k \right)_{k\in \N}$ will be a natural candidate to be equal to the
value function $V_t$.
\begin{lemma}
 \label{lemma:unifconv}
 Fix $t\in \ce{0,T}$. Let $\np{\func^k}_{k\in \N}$ be a monotonic sequence in
 $\Funcbb_t$ such that there exists $\func_1, \func_2 \in \Funcbb_t$ satisfying
 for every $k\in \N$ \( \func_1 \leq \func^k \leq \func_2.  \) Then, the
 sequence $\np{\func^k}_{k\in \N}$ converges uniformly on every compact set
 included in $\dom\np{V_t}$ to a function $\func^* \in \Funcbb_t$.
\end{lemma}
\begin{proof}
 The proof relies on the Arzela-Ascoli~theorem~\cite[Theorem
  2.13.30~p.347]{Sc1995}. Since $\phi_1$ and $\phi_2$ belong to $\Funcbb_t$,
 they are finite on $\dom\np{V_t}$. Then, the sequence of functions
 $\np{\func^k}_{k\in \N}$ is monotonic and bounded, so it converges pointwise
 on $\dom\np{V_t}$ to a limit function $\func^*$. By
 Assumption~\ref{Assumptions}-\eqref{Stability-pointwise-convergence}, this
 implies that $\func^* \in \Funcbb_t$.

 Now, fix a compact set $K\subset \dom\np{V_t}$. First, since
 $\np{\func^k}_{k\in \N} \subset \Funcbb_t$, we have that for every $k\in \N$,
 $\dom\np{\func^k}$ contains $\dom\np{V_t}$ and the sequence
 $\np{\func^k}_{k \in \N}$ share a common modulus of continuity. Second, the
 continuous functions $\lvert \func_1 \rvert$ and $\lvert \func_2 \rvert$ are
 bounded from above by quantities independent of $x$ on the compact $K$, thus,
 $\sup_{k\in \N} \sup_{x\in K} \lvert \func^k\np{x} \rvert$ is finite. Hence,
 by Arzela-Ascoli~theorem the monotonic sequence $\np{\func^k}_{k\in \N}$
 converges uniformly on the compact $K$ to the continuous function $\func^*$.
\end{proof}

\begin{proposition}[Existence of an approximating limit]
 \label{ExistenceLimit}
 Let $t\in \ce{0,T}$ be fixed, the sequence of
 functions $\left( V_t^k \right)_{k\in \N}$ defined by Equation~\eqref{Vktdef} and Algorithm~\ref{\NomAlgo} $\Pro$-a.s.\ converges
 uniformly on every compact set included in the domain of $V_t$ (solution of Equation~\eqref{DynamicProgramming}) to a function $V_t^* \in \Funcbb_t$.
\end{proposition}

\begin{proof}
 By Lemma~\ref{ProprietesFaciles}-\eqref{proofitema}, for every $k\geq 1$ we
 have that \( V_t^1 \leq V_t^k \leq V_t \), when $\optop = \sup$ (and the
 inequalities are reversed when $\optop = \inf$). Now, we have that
 $V_t^1 \in \Funcbb_t$ and by Lemma~\ref{ContinuiteVt}, the mapping $V_t$ is also in
 $\Funcbb_t$. Moreover, by Lemma~\ref{ProprietesFaciles}-\eqref{proofitema},
 the sequence $\np{V_t^k}_{k \geq 1}$ is monotonic. Thus, by
 Lemma~\ref{lemma:unifconv}, we have that $\np{V_t^k}_{k\geq 1}$ converges uniformly
 on every compact set included in $\dom\np{V_t}$ to a function
 $V_t^* \in \Funcbb_t$.

 \medskip
 This ends the proof.
\end{proof}

\begin{proposition}
 \label{ConvergenceBellmanImages}
 Let $t\in \ce{0,T-1}$ be fixed and $V_{t+1}^*$ be the function defined in Proposition~\ref{ExistenceLimit}. The sequence $\B_t \left( V_{t+1}^k \right)$ $\Pro$-a.s.\
 converges uniformly to the continuous function $\B_t \left(V_{t+1}^*\right)$
 on every compact sets included in the domain of $V_t$.
\end{proposition}

\begin{proof}
 First we consider the case $\optop = \inf$.  As the sequence
 $\left( V_{t+1}^k \right)_{k\in \N}$ is non-increasing and using the fact that the operator $\B_t$ is order preserving, the sequence
 $\np{ \B_t \np{ V_{t+1}^k }}_{k\in \N}$ is also non-increasing.
 Moreover, we have that
 \begin{align*}
  V_t^1 & \geq V_t^k \tag{\text{Lemma~\ref{ProprietesFaciles}-\eqref{proofitema}}}              \\
        & \geq \B_t\np{V_{t+1}^k} \tag{\text{Lemma~\ref{ProprietesFaciles}-\eqref{proofitemb}}} \\
        & \geq \B_t \np{V_{t+1}}  \tag{\text{Lemma~\ref{ProprietesFaciles}-\eqref{proofitema}}} \\
        & = V_t.
 \end{align*}
 Thus, by Lemma~\ref{lemma:unifconv}, the sequence of functions $\np{\B_t\np{V_{t+1}^k}}_{k \geq 1}$ converges uniformly on every compact set included in $\dom\np{V_t}$ to a function $\phi \in \Funcbb_t$. Let $K_t$ be a given compact set
 included in $\dom \np{V_t}$. We now show that the function $\phi$ is equal to $\B_t\left( V_{t+1}^*\right)$ on the given compact $K_t$ or equivalently
 we show that $\phi + \indi{K_t} = \B_t\left( V_{t+1}^* \right) + \indi{K_t}$. As already shown in Proposition~\ref{ExistenceLimit},
 we have that $V_{t+1}^k  \geq V_{t+1}^*$, which combined with the fact that the operator $\B_t$ is order preserving,
 gives, for every $k\geq 1$, that
 \(
 \B_t \np{ V_{t+1}^k }  \geq \B_t\np{V_{t+1}^*}.
 \)
 Now, adding on both side of the previous inequality the mapping $\indi{K_t}$ and taking the limit as $k$ goes to infinity,
 we have that
 \begin{equation*}
  \phi + \indi{K_t} \geq \B_t\left( V_{t+1}^* \right) + \indi{K_t}.
 \end{equation*}
 For the converse inequality, start by recalling that,
 by the compactness condition
 (see Assumption~\ref{Assumptions}-\eqref{optimal-sets}), there exists a compact set $K_{t+1} \subset \dom \np{V_{t+1}}$ such that, for every $\func \in \Funcbb_{t+1}$ and every $\lambda \geq 0$, we have that
 \begin{equation}
  \label{eq:tmp111}
  \B_t\left( \func + \lambda + \indi{K_{t+1}} \right) \leq \B_t \left( \func + \lambda \right) + \indi{K_t}.
 \end{equation}
 Now, by Proposition~\ref{ExistenceLimit}, the non-increasing sequence $\left(V_{t+1}^k \right)_{k\in \N}$ converges
 uniformly to $V_{t+1}^*\in \Funcbb_{t+1}$ on the compact set $K_{t+1}$. Thus,
 for any fixed $\epsilon > 0$, there exists an integer $k_0 \in \N$,
 such that we have
 \[
  V_{t+1}^k \leq V_{t+1}^k + \indi{K_{t+1}} \leq V_{t+1}^* + \epsilon + \indi{K_{t+1}}\,,
 \]
 for all $k \geq k_0$.
 By~Assumption~\ref{Assumptions}-\eqref{order-preserving} and~Assumption~\ref{Assumptions}-\eqref{Additively-subhomogeneous}, the operator $\B_t$ is order preserving and additively $M_t$-subhomogeneous, thus we get using Equation~\eqref{eq:tmp111} that
 \begin{align*}
  \B_t\left( V_{t+1}^k \right) \leq \B_t\left ( V_{t+1}^k + \indi{K_{t+1}} \right)
   & \leq  \B_t\np{V_{t+1}^*   + \epsilon +\indi{K_{t+1}}},         \tag{by Assumption~\ref{Assumptions}-\eqref{order-preserving}}     \\ 
   & \leq  \B_t\np{V_{t+1}^* +  \epsilon} +  \indi{K_{t}}, \tag{by~Equation~\eqref{eq:tmp111}}                                         \\
   & \leq \B_t\np{V_{t+1}^*} + M_t\epsilon + \indi{K_t} \tag{by~Assumption~\ref{Assumptions}-\eqref{Additively-subhomogeneous}}\eqfinp 
 \end{align*}
 Adding $\indi{K_t}$ on the left hand side, we have for every $k\geq k_0$ that
 \(
 \B_t\left(V_{t+1}^k\right) + \indi{K_t} \leq \B_t\left( V_{t+1}^* \right) + M_t\epsilon + \indi{K_t} .
 \)
 Thus, taking the limit when $k$ goes to infinity we obtain that
 \begin{equation*}
  \phi + \indi{K_t} \leq \B_t\left( V_{t+1}^* \right) + M_t \epsilon + \indi{K_t} .
 \end{equation*}
 The result has been proved for all $\epsilon >0$ and we have thus shown that
 $\phi = \B_t\left( V_{t+1}^*\right)$ on the compact set $K_t$. We conclude that
 $\left( \B_t\left( V_{t+1}^k\right) \right)_{k\in \N}$ converges uniformly to
 the function $\B_t\left( V_{t+1}^*\right)$ on the compact set $K_t$.
 For the case $\optop = \sup$, \emph{mutatis mutantis} we have that
 \(
 \B_t\np{V_{t+1}^k} \leq \B_t\np{V_{t+1}^*}.
 \)
 Similarly, as the sequence $\np{V_{t+1}^k}$ is non-decreasing and $\B_t$ is
 order preserving, one gets that for every $k$ large enough
 \[
  \B_t\np{V_{t+1}^*} \geq \B_t\np{V_{t+1}^*} + \indi{K_{t+1}} \geq \B_t\np{V_{t+1}^* + \epsilon + \indi{K_{t+1}}}.
 \]
 Thus, by Equation~\eqref{eq:tmp111} and $M_t$-sub-homogeneity we have that
 \(
 \B_t\np{V_{t+1}^*} + \indi{K_t} \leq \B_t\np{V_{t+1}^k} + M_t \epsilon + \indi{K_t},
 \)
 which yields the result when $k$ goes to infinity. This ends the proof.
\end{proof}

We want to exploit the fact that our approximations of the final cost function
are exact in the sense that we have equality between $V_T^k$ and $V_T$ at the
points drawn in Algorithm~\ref{\NomAlgo}, that is, the tightness assumption of the
selection function is much stronger at time $T$ than for times $t<T$. Thus we
want to propagate the information backward in time: starting from time $t=T$ we
want to deduce information on the approximations for times $t<T$.

In order to show that $V_t = V_t^*$ on some set $S_t$, a dissymmetry between upper and lower approximations is emphasized.  We introduce the notion of optimal sets $\left( S_t \right)_{t\in \ce{0,T}}$ with respect to a sequence of functions 	$\left(\func_t\right)_{t\in \ce{0,T}}$ as a condition on the sets $\left( S_t \right)_{t\in \ce{0,T}}$ such that in order to compute the restriction of $\B_t \left( \phi_{t+1} \right)$ to $S_t$, one only needs to know $\phi_{t+1}$ on the set $S_{t+1}$. The Figure~\ref{fig:OptimalSets} illustrates this notion.

\begin{definition}[Optimal sets]
 \label{optimalDraws}
 Let $\left(\func_t\right)_{t\in \ce{0,T}}$ be $T{+}1$ functions on $\X$. A
 sequence of sets $\left( S_t \right)_{t\in \ce{0,T}}$ is said to be
 \emph{$\left(\func_t\right)$-optimal} if for every $t\in \ce{0,T-1}$, we have
 \begin{equation}
  \label{eq:optimalSets}
  \B_t\left( \func_{t+1} + \indi{S_{t+1}} \right) + \indi{S_t} = \B_t \left( \func_{t+1} \right) + \indi{S_t}.
 \end{equation}
\end{definition}

\begin{figure}
 \centering
 \includegraphics[scale=0.8]{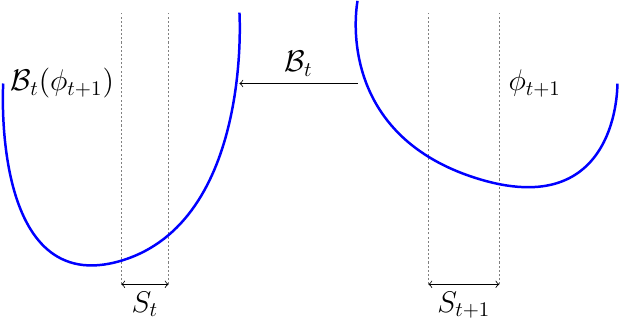}
 \caption{\label{fig:OptimalSets}The optimality of the sets
 $\left( S_t \right)_{t\in \ce{0,T}}$ means that in order to compute the
 restriction of $\B_t \left( \phi_{t+1} \right)$ on $S_t$, one only needs to
 know the values of $\phi_{t+1}$ on the set $S_{t+1}$.}
\end{figure}

When approximating from below, the optimality of sets is only needed for the
limit functions $\left(V_t^*\right)_{t\in \ce{0,T}}$, whereas when approximating
from above, one needs the optimality of sets with respect to the value functions
$\left(V_t\right)_{t\in \ce{0,T}}$. It seems easier to ensure the $\np{V_t^*}$-optimality of
sets than $\np{V_t}$-optimality as the function $V_t^*$ is known through the
sequence $\left( V_t^k \right)_{k\in \N}$, whereas the function $V_t$ is,
\emph{a priori}, unknown. This fact is discussed in
Sections~\ref{SDDP_Example} and~\ref{sec:Exemples_switch}.

\begin{lemma}[Uniqueness in restricted Bellman Equations]
 \label{UnicityBellman}
 Let $\left( X_t \right)_{t \in \ce{0,T}}$ be a sequence of sets such that for
 every $t\in \ce{0,T}$, $X_t \subset \dom\np{V_t}$ and which is

 \noindent -- $\np{V_t}$-optimal when $\optop = \inf$,

 \noindent -- $\np{V_t^*}$-optimal when $\optop = \sup$.

 If the sequence of functions $\left(V_t^* \right)_{t\in \ce{0,T}}$ satisfies the following restricted Bellman Equations:
 \begin{equation}
  \label{RestrictedBellmanEquation}
  V_T^* + \delta_{X_T}= \psi + \delta_{X_T}
  \quad \text{ and }\quad
  \forall t\in \ce{0,T-1}, \ \B_t\left( V_{t+1}^* \right) + \indi{X_t} = V_{t}^*  + \indi{X_t}
  \eqfinp
 \end{equation}
 Then, for every $t\in \ce{0,T}$ and every $x\in X_t$, we have that $V_{t}^* (x) = V_t(x)$.
\end{lemma}

\begin{proof}
 We prove the lemma by backward induction on time $t\in \ce{0,T}$. We first treat the case $\optop = \inf$.
 At time $t=T$, since $V_T$ is given by Equation~\eqref{DynamicProgramming}, we have $V_T= \psi$. We therefore have
 by Equation~\eqref{RestrictedBellmanEquation} that $V_T^* + \indi{X_T} = \psi + \indi{X_T} = V_T + \indi{X_T}$, which gives the fact that
 functions $V_T^*$ and $V_T$ coincide on the set $X_T$.
 Now, let time $t\in \ce{0,T-1}$ be fixed and assume that we have $V_{t+1}^*(x) =
  V_{t+1}(x)$ for every $x\in X_{t+1}$, or equivalently:
 \begin{equation}
  \label{eq:tmp0}
  V_{t+1}^*  + \indi{X_{t+1}} = V_{t+1} + \indi{X_{t+1}}\,.
 \end{equation}
 Using Lemma~\ref{ProprietesFaciles}-\eqref{proofitema}, the sequence of functions $\np{V_t^k}_{k\in \N}$ is lower bounded by $V_t$. Taking the
 limit in $k$, we obtain that $V_t^* \geq V_t$, thus we only have to prove that
 $V_t^* \leq V_t$ on $X_t$, that is $V_t^* + \indi{X_t} \leq V_t + \indi{X_t}$.
 We successively have:
 \begin{align}
  V_{t}^*  + \indi{X_t}
   & = \B_t\left( V_{t+1}^* \right) + \indi{X_t}
  \tag{by~\eqref{RestrictedBellmanEquation}}
  \\
   & \le
  \B_t\left( V_{t+1}^* +\indi{X_{t+1}} \right) + \indi{X_t}
  \tag{$\B_t$ is order preserving}
  \\
   & = \B_t\left( V_{t+1} +\indi{X_{t+1}} \right) + \indi{X_t}
  \tag{by induction assumption~\eqref{eq:tmp0}}
  \\
   & = \B_t\left( V_{t+1} \right) + \indi{X_t}
  \tag{by~\eqref{eq:optimalSets}, $(X_t)_{{t \in \ce{0,T}}}$ is $(V_t)$-optimal}
  \\
   & = V_t + \indi{X_t},
  \tag{by~\eqref{DynamicProgramming}}
 \end{align}
 which concludes the proof in the case of $\optop = \inf$.

 Now we prove the case $\optop = \sup$ in a similar way by backward induction on time $t\in \ce{0,T}$. As for the case $\optop=\inf$, at time $t=T$, one has $V_T^* + \indi{X_T} = V_T + \indi{X_T}$. Now assume that for some $t\in \ce{0,T-1}$ one has $V_{t+1}^{*} + \indi{X_{t+1}} = V_{t+1} + \indi{X_{t+1}}$. By Lemma~\ref{ProprietesFaciles}-\eqref{proofitema}, the sequence of functions $\np{V_t^k}$ is now upper bounded by $V_t$. Thus, taking the limit in $k$ we obtain that $V_t^* \leq V_t$ and we only need to prove that $V_t^* + \indi{X_t} \geq V_t + \indi{X_t}$. We successively have:
 \begin{align}
  V_{t}  + \indi{X_t}
   & = \B_t\left( V_{t+1} \right) + \indi{X_t}
  \tag{by~\eqref{DynamicProgramming}}
  \\
   & \leq
  \B_t\left( V_{t+1} +\indi{X_{t+1}} \right) + \indi{X_t}
  \tag{$\B_t$ is order preserving}
  \\
   & = \B_t\left( V_{t+1}^* +\indi{X_{t+1}} \right) + \indi{X_t}
  \tag{by induction assumption~\eqref{eq:tmp0}}
  \\
   & = \B_t\left( V_{t+1}^* \right) + \indi{X_t}
  \tag{$(X_t)_{{t \in \ce{0,T}}}$ is $(V_t^*)$-optimal}
  \\
   & = V_t^* + \indi{X_t}, \tag{by~\eqref{RestrictedBellmanEquation}}
 \end{align}
 This ends the proof.
\end{proof}

One cannot expect the limit function, $V_t^*$, to be equal everywhere to the
value function, $V_t$, given by Equation~\eqref{DynamicProgramming}.  However,
one can expect an (almost sure over the draws) equality between the two
functions $V_t$ and $V_t^*$ on all possible cluster points of sequences
$\left( y_k \right)_{k\in \N}$ with $y_k \in K_t^k$ for all $k\in \N$, that is,
on the set $\limsup K_t^k$.

\begin{theorem}[Convergence of \NomAlgo]
 \label{ConvergenceTheorem}
 Define $K_t^* := \limsup_k K_t^k$, for every time $t\in \ce{0,T}$.  Assume
 that, $\Pro$-a.s.\ the sets $\left(K_t^*\right)_{t\in\ce{0,T}}$ are
 $\np{V_t}$-optimal when $\optop = \inf$ (resp.  $\np{V_t^*}$-optimal when
 $\optop = \sup$). Then, $\Pro$-a.s.\ for every $t \in \ce{0,T}$ the function
 $V_t^*$ defined in Proposition~\ref{ExistenceLimit} is equal to the value function $V_t$
 on $K_t^*$.
\end{theorem}

\begin{proof} We will only consider the case $\optop = \inf$ as the proof for the case $\optop = \sup$ is analogous.
 We will show that Equation~\eqref{RestrictedBellmanEquation} holds $\Pro$-almost surely with $X_t = K_t^*$, $t\in \ce{0,T}$.
 The proof is decomposed in several steps

 \noindent $\bullet$ Reformulation using the separability of $\X$.
 Let $C := (C_t)_t \subset \X^{T+1}$ be compact in $\dom(V_0)\times \ldots \times \dom(V_T)$. For every $t\in \ce{0,T-1}$, set $\Delta_t : x_t \in \X \to V_t^*(x_t) - \B_t\np{V_{t+1}^*}(x_t) \in \overline{\R}$, $\Delta_T : x_T \in \X \to V_T^*(x_T) - V_T(x_T) \in \overline{\R}$ and $\Delta := \np{\Delta}_{t\in \ce{0,T}}$. Also write $K^* := \np{K_t^*}_{t\in \ce{0,T}}$. We want to show that
 \begin{equation}
  \label{eq:thm1}
  \Pro\Bc{\forall x \in C, \bp{x \in K^* \Rightarrow \Delta(x) = 0} } = 1 \; .
 \end{equation}
 By continuity of $V_t^* - \B_t\np{V_{t+1}^*}$ (resp. $V_T^* - V_T$) for every $t\in \ce{0,T-1}$ (resp. $t = T$) and compactness of $K$, Equation~\eqref{eq:thm1} is equivalent to
 \begin{equation}
  \label{eq:thm2}
  \Pro \Bc{\forall \epsilon > 0, \exists r > 0, \forall x \in K, \bp{x \in \np{K^* + r \mathbb{B}} \Rightarrow \Delta(x) \leq \epsilon }} = 1 \; .
 \end{equation}
 Without loss of generality, by density, one may restrict $\epsilon$ and $r$ to the countable set $\Q_+^*$ and the set $C$ to the set $C \cap \np{\Q^n}^{T+1}$, that is, Equation~\eqref{eq:thm2} is equivalent to
 \begin{equation}
  \label{eq:thm3}
  \forall \epsilon \in \Q_+^*, \exists r \in \Q_+^*, \forall x \in C \cap \np{\Q^n}^{T+1},
  \Pro\bc{  x \in \np{K^* + r\mathbb{B}} \Rightarrow \Delta(x) \leq \epsilon } = 1     \;.
 \end{equation}

 For the remainder of the proof, we fix $\epsilon \in \Q_+^*$. Now, we exploit
 the equicontinuity of the sequence of functions $\np{V_t^k}_{k\in \N}$ and
 $\bp{\B_t\np{V_{t+1}^k}}_{k\in \N}$ in order to compute a suitable radius
 $r' \in \Q_+^*$ so as to satisfy Equation~\eqref{eq:thm3}. We separate the cases $t=T$
 and $t<T$.

 \noindent $\bullet$ Equicontinuity and uniform convergence, case
 $\mathbf{t = T}$.  As the functions $V_T$ and $V_t^k$, for $k\in \N$ are in
 $\Funcbb_T$, they share a common modulus of continuity on the compact
 $C_T$. Thus, they share a common uniform modulus of continuity. Hence, there
 exists a radius $r_T \in \Q_+^*$, such that for every $x_T \in C_T \cap \Q^n$,
 if $y_T \in B\np{x_T, r_T} \cap \dom{V_T}$, then
 \begin{equation}
  \label{eq:thm4}
  \big\lvert V_T^{k+1}\np{x_T} - V_T^{k+1}(y_T) \rvert \leq \frac{\epsilon}{3} \ \text{and} \ \lvert V_T\np{y_T} - V_T(x_T) \rvert \leq \frac{\epsilon}{3} \; .
 \end{equation}
 Now, as $\np{V_T^k}_{k\in \N}$ converges uniformly to $V_t^*$ on the compact $C_t \subset \dom V_T$, there exists a rank $k_T \in \N$ such that, if $k\geq k_T$, then for all $x_T \in K_T$,
 \begin{equation}
  \label{eq:thm5}
  \lvert V_T^{k+1}\np{x_T} - V_T^*\np{x_T} \rvert \leq \frac{\epsilon}{3}.
 \end{equation}

 \noindent $\bullet$ Equicontinuity and uniform convergence, case
 $\mathbf{t\in \ce{0,T-1}}$.  The sequences $\np{V_t^k}_{k\in \N}$,
 resp. $\np{\B_t\np{V_{t+1}^k}}_{k\in \N}$ are uniformly equicontinuous on the
 compact $C_t \subset \dom\np{V_t}$. There exists a radius $r_t \in \Q_+^*$ such
 that for every $x_t \in C_t$, if $y_t \in B\np{x_t,r_t} \cap \dom V_t$, then
 for every $k\in \N$,
 \begin{equation}
  \label{eq:thm6}
  \lvert V_t^{k+1}(x_t) - V_t^{k+1}(y_t) \rvert \leq \frac{\epsilon}{4} \quad \text{and}
  \quad \lvert \B_t\np{V_{t+1}^{k+1}(y_t)} - \B_t\np{V_{t+1}^{k+1}(x_t)} \rvert \leq \frac{\epsilon}{4} \eqfinp
 \end{equation}
 By uniform convergence of the sequence $\np{V_t^k}_{k\in \N}$ (resp. $\B_t\np{V_{t+1}^k}_{k\in \N}$) to $V_t^*$ (resp. to $\B_t\np{V_{t+1}^*}$) on the compact $C_t \subset \dom(V_t)$, there exists a rank $k_t \in \N$ such that, if $k\geq k_t$, then for every $x_t \in K$
 \begin{equation}
  \label{eq:thm7}
  \lvert V_t^*(x_t) - V_t^{k+1}(x_t) \rvert \leq \frac{\epsilon}{4} \ \text{and} \ \lvert \B_t\np{V_{t+1}^{k+1}}(x_t) - \B_t\np{V_{t+1}^*}(x_t)\rvert \leq \frac{\epsilon}{4}
 \end{equation}

 \noindent $\bullet$ There exists a draw $x_t^{k^*}$ of the sequence of trial points $\np{x_t^k}_{k\in \N}$ arbitrarily close to any given point of $K^*$.
 Throughout the remainder of the proof, we fix ranks $k_t \in \N$ and radii $r_t \in \Q_+^*$ defined in Step $2$ and set
 \[
  \overline{k} := \max_{t\in \ce{0,T}} k_t \in \N \ \text{and} \ \underline{r} := \min_{t\in \ce{0,T}} r_t \; .
 \]
 By the Trial point oracle assumption, there exists $r \in \Q_+^*$ such that, for every $x\in C$,
 \begin{equation}
  \label{eq:thm8}
  \Pro\nc{x \in K^* + r\mathbb{B} \Rightarrow x \in \limsup_{k\in \N} B\np{x^k, \underline{r}/2}} = 1 \; .
 \end{equation}
 Now, fix  $x\in C \cap \np{\Q^n}^{T+1}$. By Equation~\eqref{eq:thm8}, $\Pro$-a.s., if $x\in K^* + r\mathbb{B}$ then $x \in \limsup_{k\in \N} B\np{x^k, \underline{r}/2}$, so  by Lemma~\ref{lemmatrial}, $\np{x^k}_k \in B(x, \underline{r})$ infinitely often. Hence, $\Pro$-a.s., if $x\in K^* + r\mathbb{B}$, then there exists $k^* \geq \overline{k}$ such that
 \begin{equation}
  \label{eq:thm9}
  x^{k^*} \in B(x, \underline{r}).
 \end{equation}

 \noindent $\bullet$ Conclusion.
 When $t = T$, by triangle inequality, $\Pro$-a.s. we have that
 \begin{align*}
  \Delta(x_T)
   & \leq \underbrace{\lvert V_T^*\np{x_T} - V_T^{k^*+1}\np{x_T} \rvert}_{\leq \epsilon/3 \ \text{by \eqref{eq:thm5} and \eqref{eq:thm9}} } + \underbrace{\lvert V_T^{k^*+1}\np{x_T} - V_T^{k^*+1}\np{x_T^{k^*}}\rvert}_{\leq \epsilon/3 \ \text{by \eqref{eq:thm4}}}
  \\
   & \quad + \underbrace{\lvert V_T^{k^*+1}\np{x_T^{k^*}} - V_T\np{x_T^{k^*}} \rvert}_{= 0 \ \text{by Tightness Lemma~\ref{ProprietesFaciles}-\eqref{proofitemd}}}                                                                                                    
  + \underbrace{\lvert V_T\np{x_T^{k^*}} - V_T\np{x_T} \rvert}_{\leq \epsilon/3 \ \text{by \eqref{eq:thm5} and \eqref{eq:thm9}} }
  \\
   & \leq \epsilon \; .
 \end{align*}
 When $t \in \ce{0,T-1}$, by triangle inequality, $\Pro$-a.s. we have that
 \begin{align*}
  \Delta(x_t)
   & \leq \underbrace{\lvert V_t^*(x_t) - V_t^{k+1}(x_t) \rvert}_{\leq \epsilon/4 \ \text{by \eqref{eq:thm7}}} + \underbrace{\lvert V_t^{k+1}(x_t) - V_t^{k+1}(x_t^k) \rvert}_{\leq \epsilon/4 \ \text{by \eqref{eq:thm6} and \eqref{eq:thm9}} }
  \\
   & \quad + \underbrace{\lvert V_t^{k+1}(x_t^k) - \B_t\np{V_{t+1}^{k+1}}(x_t^k)\rvert}_{= \ 0 \ \text{by Lemma~\ref{ProprietesFaciles}-\eqref{proofitemc}}}
  \\
   & \quad + \underbrace{\lvert  \B_t\np{V_{t+1}^{k+1}}(x_t^k) - \B_t\np{V_{t+1}^{k+1}}(x_t) \rvert}_{\leq \epsilon/4 \ \text{by \eqref{eq:thm6} and \eqref{eq:thm9}} } + \underbrace{\lvert \B_t\np{V_{t+1}^{k+1}}(x_t)- \B_t\np{V_{t+1}^{*}}(x_t) \rvert}_{\leq \epsilon/4 \ \text{by \eqref{eq:thm7}} }
  \\
   & \leq \epsilon \; .
 \end{align*}
 Thus, we have shown Equation~\eqref{eq:thm3}, \emph{i.e.} $\Pro$-a.s., for every $t\in \ce{0,T}$ we have $V_t^* = \B_t\np{V_{t+1}^*}$ on $K_t^*$.
 The sequence $\left( V_t^* \right)_{k\in \N}$
 satisfies the restricted Bellman Equation~\eqref{RestrictedBellmanEquation} with the sequence
 $\left( K_t^* \right)_{k\in \N}$. The conclusion follows from the Uniqueness lemma (Lemma~\ref{UnicityBellman}).
\end{proof}



\section{SDDP selection function: lower approximations in the linear-convex framework}
\label{SDDP_Example}
We will show that our framework contains a similar framework of (the
deterministic version of) the SDDP algorithm as described in \cite{Gi.Le.Ph2015}
and yields the same result of convergence. Let $\X = \R^n $ be a continuous
state space and $\U = \R^m$ a continuous control space. We want to solve the
following problem
\begin{equation}
 \label{pb:linear-convex}
 \begin{aligned}
  \min_{\substack{x=\np{x_0, \ldots, x_T}                                 \\u = \np{u_0, \ldots u_{T-1}}}} & \sum_{t=0}^{T-1} c_t (x_t, u_t) + \psi(x_T) \\ \text{s.t.} \  & x_0 \in \X \ \text{is given,}                        \\
   & \forall t \in \ce{0,T}, \ x_t \in  \X,                               \\
   & \forall t \in \ce{0,T-1}, \ u_t \in \U,  \, x_{t+1} = f_t(x_t, u_t).
 \end{aligned}
\end{equation}

We make similar assumptions as in the literature of SDDP (\emph{e.g.}~\cite{Gi.Le.Ph2015}), note that in our formulation, we have put the constraints on the states and controls on the cost functions. We refer to~\cite{Au.Ek1984} and~\cite{Ro.We2009} for results on set-valued mappings.

\begin{assumption}
 \label{hypo:linear-convex}
 For all $t \in \ce{0,T-1}$ we assume that:
 \begin{enumerate}
  \item The dynamic $f_t : \X \times \U \longrightarrow \X$ is linear, $f_t(x, u) = A_t x + B_t u$, for some given matrices $A_t$ and $B_t$ of compatible dimensions.
  \item\label{hypo:convexcosts} The cost function $c_t : \X \times \U \longrightarrow \overline{\R}$ is a proper lower semicontinuous (l.s.c.) convex function which is $L_{c_t}$-Lipschitz continuous on its (convex) domain, $\dom\np{c_t}$.
  \item The projection on $\X$ of $\dom\np{c_t}$, denoted $X_t$, is a convex polytope with non-empty interior.
  \item\label{hypo:multiapplication} Define the set-valued mapping $U_t : \X \rightrightarrows \U$, for every $x\in \X$
        \[
         U_t\np{x} := \left\{ u \in \U \mid \np{x,u} \in \dom\np{c_t} \ \text{and} \ f_t\np{x,u} \in X_{t+1}\right\},
        \]
        where we assume that

        \noindent $\bullet$ For every $x\in X_t$, $U_t(x)$ is compact.

        \noindent $\bullet$ The graph of the set-valued mapping $U_t$ has a non-empty interior.

        \noindent $\bullet$ For every $x\in X_t$, there exists $u \in U_t(x)$.\footnote{known as a \emph{Relatively Complete Recourse} assumption.}

        \noindent $\bullet$ The set-valued mapping $U_t$ is $L_{U_t}$-\emph{Lipschitz continuous}\footnote{For all $x,x' \in \X$, $U_t(x') \subset U_t(x) + L_{U_t} \lVert x' - x \rVert$.} (hence, both upper and lower semicontinuous).
 \end{enumerate}
 Moreover, at time $t = T$, we assume that $X_T := \dom\np{V_T}\subset \X$ is
 convex and compact with non-empty interior, the final cost function
 $\psi : \X \longrightarrow \overline{\R}$ is a proper convex l.s.c. function with
 known compact convex domain and $\psi$ is $C_T$-Lipschitz continuous on its
 domain.
\end{assumption}

\begin{remark}
 Under Assumption~\ref{hypo:linear-convex}, the graph of the set-valued mapping $U_t$ is convex, and its domain is $X_t$ by the RCR assumption.
\end{remark}

\begin{remark}
 A sufficient condition to ensure that the set-valued mapping $U_t$ is Lipschitz
 continuous is given in~\cite[Example 9.35]{Ro.We2009}: $U_t$ is Lipschitz when
 its graph is convex polyhedral, which is the classical framework of
 SDDP. Moreover a Lipschitz constant can be explicitly computed.
\end{remark}

For every time step $t\in \ce{0,T-1}$, recall the \emph{Bellman operator}
$\B_t$ for every function $\func : \X \rightarrow \overline{\R}$ by:
\begin{equation}
 \label{Bellman:linear-convex}
 \B_t(\func ) := \inf_{u \in \U}\Bp{ c_t \np{\cdot, u} + \func\bp{f_t\np{\cdot, u}}}  \eqfinp
\end{equation}
Moreover, for every function $\func : \X \rightarrow \overline{\R}$ and every $\np{ x, u } \in \X \times \U$ we define
\begin{equation}
 \B_t^u\left( \func \right)(x) := c_t\left(x,u\right) + \func\bp{f_t\np{x, u}} \in \overline{\R} \eqfinp
\end{equation}
The Bellman equations of Problem \eqref{pb:linear-convex} can be written using
the Bellman operators $\B_t$ given by Equation~\eqref{Bellman:linear-convex}:
\begin{equation}
 \label{DP:linear-convex}
 V_T  = \psi \quad\text{and}\quad
 \forall t\in \ce{0,T{-}1}, V_t: x \in \X \mapsto  \B_t(V_{t+1})(x) \in \overline{\R}\eqfinp
\end{equation}

In Proposition~\ref{SDDP:stability}, we establish a stability property of the Bellman
operators given by Equation~\eqref{Bellman:linear-convex}. The image of a Lipschitz
continuous function by the operator $\B_t$ will also be Lipschitz continuous and we
give an explicit (conservative) Lipschitz constant.

\begin{proposition}
 \label{SDDP:stability}
 Under Assumption~\ref{hypo:linear-convex}, for every $t\in \ce{0,T-1}$, given a constant
 $L_{t+1} > 0$, there exist a constant $L_t > 0$ such that if
 $\func : \X \to \overline{\R}$ is convex l.s.c. proper with domain $X_{t+1}$ and $L_{t+1}$-Lipschitz continuous on
 $X_{t+1}$ then $\B_t\left( \func \right)$ is convex l.s.c. proper with domain $X_t$ and $L_t$-Lipschitz continuous on
 $X_t$.
\end{proposition}

\begin{proof}[Proof of Proposition~\ref{SDDP:stability}]
 Fix $t\in \ce{0,T-1}$ and let $\func : \X \to \overline{\R}$ be a convex l.s.c.
 proper function with domain $X_{t+1}$ and $L_{t+1}$-Lipschitz continuous
 function on $X_{t+1}$. We show that $\dom\bp{\B_t\np{\func}} = X_t$. Let
 $x_t \in X_t$ be arbitrary. By the RCR Assumption, there exist
 $u_t \in U_t\np{x_t}$ such that $f_t \left( x_t, u_t \right) \in X_{t+1}$ and
 $\np{x_t,u_t} \in \dom\np{c_t}$. As the domain of $\func$ is $X_{t+1}$, we
 have that
 \[
  \inf_{u \in \U} \Bp{c_t\np{x_t,u} + \func\bp{f_t\np{x_t,u}}} \leq c_t\left(x_t,u_t\right) + \func \bp{f_t\np{x_t,u_t}} < +\infty \eqfinp
 \]
 Thus, we have shown that $\dom \bp{ \B_t\np{ \func}}$ includes $X_t$. Conversely, if $x\notin X_t$, then for every $u\in \U$, we have $c_t\np{x,u} = +\infty$, hence $x\notin \dom\np{\B_t\np{\func}}$. This implies that $\dom\bp{\B_t\np{\func}} \subset X_t$ and the equality follows.

 Moreover, the above infimum can be restricted to $U_t\np{x}$, which is compact. As the function $x \mapsto \B_t\left( \func \right)(x)$ is convex (resp. l.s.c.) on $X_{t}$ as $\np{x,u} \mapsto  c_t\left(x,u\right) + \func\left(f_t\left(x,u\right)\right)$ is jointly convex (resp. l.s.c. and $U_t(x)$ is compact).

 Since $c_t(x, \cdot)$, $\func$ are l.s.c. and $f_t\np{x, \cdot}$ is continuous, the above infimum is attained. We will denote by $u_x \in U_t\np{x}$ a minimizer, note that $f_t\np{x,u_x} \in X_{t+1}$.

 We finally show that the function $\B_t\left( \func \right)$ is Lipschitz on $X_t$ with a constant $L_t> 0$ that only depends on the data of Problem \eqref{pb:linear-convex}. Fix $x, x' \in X_t$ and denote by $u_{x'} \in U_t\np{x'}$ an optimal control at $x'$, \emph{i.e.} $\B_t^{u_{x'}}\np{\func}\np{x'}= \B_t\np{\func}\np{x'}$. For every $u \in U_t\np{x}$, we have that
 \begin{align}
  \B_t\np{\func}\np{x}
   & \leq \B_t\np{\func}\np{x'} + \B_t^u\np{\func}\np{x} - \B_t\np{\func}\np{x'}
  \notag
  \\
   & = \B_t\np{\func}\np{x'} + \bp{c_t\np{x,u} - c_t\np{x',u_{x'}}} + \Bp{\func\bp{f_t\np{x,u}} - \func\bp{f_t\np{x',u_{x'}}}}
  \notag
  \\
   & \leq \B_t\np{\func}\np{x'} + L_{c_t}\np{ \lVert x - x' \rVert + \lVert u - u_{x'} \rVert}
  \label{tmp123456789}                                                                                                               \\
   & \quad + L_{t+1} \Bp{\lambdamax\np{A_t^TA_t}^{1/2}\lVert x - x' \rVert + \lambdamax\np{B_t^TB_t}^{1/2}\lVert u - u_{x'} \rVert }
  \notag.
 \end{align}
 Indeed, as the domain of $c_t(x, \cdot)$ is $U_t(x)$, the domain of $\func$ is $X_{t+1}$ and that for every $u\in U_t(x)$, we have $f_t\np{x,u} \in X_{t+1}$, Equation~\eqref{tmp123456789} holds for every $u\in U_t(x)$.

 Now, we will bound from above $\lVert u - u_{x'} \rVert$ by $\lVert x - x' \rVert$ multiplied by a constant.
 By Assumption~\ref{hypo:linear-convex}-\eqref{hypo:multiapplication} the set-valued mapping $U_t$ is $L_{U_t}$-Lipschitz on its domain $X_t$. Hence, by definition, there exists $\tilde{u} \in U_t\np{x}$ such that:
 \begin{equation}
  \lVert \tilde{u} - u_{x'} \rVert \leq L_{U_t} \lVert x - x' \rVert. \label{tmp:LipschitzIntersection}
 \end{equation}
 Replacing $u$ by $\tilde{u}$ in Equation~\eqref{tmp123456789}, by Equation~\eqref{tmp:LipschitzIntersection} we deduce that
 \(
 \B_t\np{\func}\np{x} - \B_t\np{\func}\np{x'} \leq L_{t} \lVert x - x' \rVert,
 \)
 where the Lipschitz constant $L_{t}>0$ only depends on the data of Problem \eqref{pb:linear-convex}. \emph{Mutatis mutandis}, we have that
 \(
 \B_t\np{\func}\np{x'} - \B_t\np{\func}(x) \leq L_{t} \lVert x - x' \rVert,
 \)
 and the result follows.
\end{proof}

\begin{remark}
 Knowing the value function at time $t = T$, by
 Proposition~\ref{SDDP:stability} we can compute recursively backward in time
 the domain of $V_t$ for each $t < T$: it is equal to the projection on $\X$ of
 the domain of the cost function, which is $X_t$ and known to the decision
 maker. Moreover using Proposition~\ref{SDDP:stability} we have that, for every
 $t$, the value function $V_t$ is convex l.s.c. proper and Lipschitz continuous on
 its domain, with a computable constant.
\end{remark}

As lower semicontinuous proper convex functions can be approximated by a supremum of affine function, for every $t\in \ce{0,T}$ we define $\Funcb_t^{\SDDP}$ to be the set of affine functions $\func : x \in \X \mapsto \langle a, x \rangle + b \in \R$, $a \in \X$, $b\in \R$ with $\Vert a \Vert_2 \leq L_t$ if $x \in X_t$ and $+\infty$ otherwise. Moreover, we shall denote by $\Funcbb_t^{\SDDP}$ the set of convex functions $\func : \X \mapsto \overline{\R}$ which are $L_t$-Lipschitz continuous on $X_t$, of domain $X_t$ and proper.

\begin{proposition}
 \label{propStructurale_SDDP}
 Under Assumption~\ref{hypo:linear-convex}, the Problem~\ref{pb:linear-convex} and the
 Bellman operators defined in Equation~\eqref{DP:linear-convex} satisfy the structural
 assumptions given in Assumption~\ref{Assumptions}.
\end{proposition}

\begin{proof} We prove successively each assumption listed in Assumption~\ref{Assumptions}.

 \noindent
 $\bullet$\ref{Assumptions}-\eqref{Stability-pointwise-optimum}.
 Recall that we are here on the case $\optop = \sup$.
 Fix $t\in \ce{0,T}$ and let $\Func \subset \Funcb^{\SDDP}_t$ be a set of
 affine $L_t$-Lipschitz continuous functions with domain $X_t$.  For every
 $x, x' \in X_t$, we have that
 \[
  \vert \vopt_\Func(x) - \vopt_\Func\left(x'\right) \vert
  = \vert \sup_{\func \in F} \func(x) - \sup_{\func \in F} \func(x') \vert
  \leq \sup_{\func \in F} \vert \func(x) - \func(x') \vert \leq L_t\Vert x - x' \Vert.
 \]
 Thus, the function $\vopt_{\Func}$ is $L_t$-Lipschitz continuous. As a supremum of affine functions is convex and l.s.c.,
 $\vopt_\Func$ is also convex and l.s.c., we have thus shown that $\vopt_\Func \in \Funcbb_{t}^{\SDDP}$.

 \noindent $\bullet$\ref{Assumptions}-\eqref{Stability-pointwise-convergence} and
 \ref{Assumptions}-\eqref{CommonRegularity}.
 By construction, for all $t\in \ce{0,T}$, every element of $\Funcbb_t^{\SDDP}$ is $L_t$-Lipschitz continuous. Thus, by the previous point, $\Funcbb_t^{\SDDP}$ is also stable by pointwise convergence.

 \noindent $\bullet$\ref{Assumptions}-\eqref{Final-condition}. 
 As $\psi$ is convex proper and $L_T$-Lipschitz continuous on $X_T$, it is a countable (as $\R^n$ is separable) supremum of $L_T$-Lipschitz affine functions.

 \noindent $\bullet$\ref{Assumptions}-\eqref{StabilityBellman}. 
 This has been shown in Proposition~\ref{SDDP:stability}.

 \noindent $\bullet$\ref{Assumptions}-\eqref{order-preserving}. 
 Let $\phi_1$ and
 $\phi_2$ be two functions over $\X$ such that $\phi_1 \leq \phi_2$ \emph{i.e.}
 for every $x\in \X$, we have $\phi_1(x) \leq \phi_2(x)$. We want to show that
 $\B_t\left( \phi_1 \right) \leq \B_t\left(\phi_2 \right)$. Let $x\in \X$, we
 have:
 \begin{align*}
  \B_t\left( \phi_1\right)(x)
   & = \inf_{u\in \U} c_t(x,u) + \phi_1\left( f_t(x,u) \right)     \\
   & \leq \inf_{u\in \U} c_t(x,u) + \phi_2 \left( f_t(x,u) \right) \\
   & = \B_t \left( \phi_2 \right)(x).
 \end{align*}
 \noindent $\bullet$\ref{Assumptions}-\eqref{Additively-subhomogeneous}. 
 We will
 show that $\B_t$ is additively homogeneous, hence one can choose $M_t=1$ in
 Assumption~\ref{Assumptions}-\eqref{Additively-subhomogeneous}.  Let $\lambda \in \R$ be
 a given constant and $\func$ a given function in $\Funcbb_{t+1}$. We identify
 the constant $\lambda$ with the constant function
 $\lambda : x \mapsto \lambda$ and we have for all $x\in \X$:
 \begin{align*}
  \B_t\left( \lambda + \phi \right)(x)
   & = \inf_{u\in \U} \Bp{c_t(x,u) + \np{\lambda + \phi}\bp{f_t(x,u)}}
  \\
   & = \inf_{u\in \U} \Bp{ c_t(x,u) +  \lambda + \phi\bp{f_t(x,u)}}
  \\
   & = \lambda + \inf_{u\in \U} \Bp{ c_t(x,u) + \phi\bp{f_t(x,u)}}
  \\
   & = \lambda + \B_t\np{\phi}(x).
 \end{align*}

 \noindent $\bullet$ \ref{Assumptions}-\eqref{proper-value}. 
 By backward recursion on
 time step $t \in \ce{0,T}$ and by Proposition~\ref{SDDP:stability}, for every time step
 $t\in \ce{0,T}$ the function $V_t$ given by the Dynamic Programming
 Equation~\eqref{DP:linear-convex} is convex and $L_t$-Lipschitz continuous on $X_t$.

 \noindent $\bullet$\ref{Assumptions}-\eqref{optimal-sets}. 
 Fix $t\in \ce{0,T-1}$, an
 arbitrary element $\func \in \Funcbb_t^{\SDDP}$, a constant
 $\lambda \geq 0$ and set $\tilde{\func} := \func + \lambda$. We will show that
 for every compact set $K_t \subset X_t$, there exist a compact set
 $K_{t+1} \subset X_{t+1}$ such that
 \begin{equation}
  \label{tmpwololo}
  \B_t \left( \tilde{\func} + \indi{K_{t+1}} \right) + \indi{K_t} = \B_t \left( \tilde{\func} \right) + \indi{K_t},
 \end{equation}
 which will imply the desired result.
 Now, Equation~\eqref{tmpwololo} is equivalent to the fact that for every state $x_t \in K_t$, there exist a control $u_t \in U_t\np{x}$ such that
 \begin{equation*}
  f_t \np{ x_t, u_t } \in K_{t+1}
  \quad\text{where}\quad
  u_t\in \argmin_{u\in U_t\np{x}} \B_t^u \np{\tilde{\func}}\np{x_t} = c_t \np{x_t, u } + \tilde{\func}\bp{ f_t \np{x_t, u}}
  \eqfinp
 \end{equation*}
 Set $K_{t+1} := f_t\np{X_t, U_t\np{X_t}}$, it satisfies Equation~\eqref{tmpwololo}, we show that it is compact. As $X_t$ is compact and $f_t$ is continuous, it is sufficient to prove that $U_t\np{X_t}$ is compact, which is true as $U_t$ is upper semicontinuous (u.s.c.) and non-empty compact valued, see~\cite[Proposition 11 p.112]{Au.Ek1984}.
 This ends the proof
\end{proof}

Now, we define a compatible selection function for $\optop=\sup$.
Let $t\in \ce{0,T-1}$ be fixed, for any $\Func \subset \Funcb_t^{\SDDP}$ and $x\in \X$, we define the following optimization problem
\begin{subequations}
 \begin{align}
  \label{NodalProblemSDDP}
  \min_{ (x',u,\lambda) \in X_t \times U_t(x') \times  \R}
   & c_t\np{ x', u} + \lambda \\
  \text{s.t.}
   & \quad
  x' = x \quad\text{and}\quad \func\bp{f_t\np{ x', u}} \leq \lambda \quad \forall \func \in \Func\eqfinp
 \end{align}
\end{subequations}

If we denote by $b$ its optimal value and by $a$ a Lagrange multiplier
associated to the constraint $x'- x = 0$ at the optimum, that is such that
$(x',u;\lambda, a)$ is a stationary point of the Lagrangian
$c_t\left( x', u\right) + \lambda -\langle a , x' - x \rangle$, then we define
\[
 \func_t^{\SDDP}\left( \Func, x \right) := x'\mapsto \langle a , x' - x \rangle + b + \indi{X_t}(x')\enspace.
\]
Finally, at time $t= T$, for any $\Func \subset \Funcb_T^{\SDDP}$ and $x\in \X$, fix $a \in \partial V_T (x)$ and define
\[
 \func_T^{\SDDP}\left( \Func, x \right) := x'\mapsto \langle a , x' - x \rangle + V_T\left( x \right).
\]

\begin{proposition}
 For every time $t\in \ce{0,T}$, the function $\func_t^{\SDDP}$ is a compatible selection function
 for $\optop=\sup$ in the sense of Definition~\ref{det:CompatibleSelection}.
\end{proposition}

\begin{proof}
 Fix $t\in \ce{0,T-1}$, $\Func \subset \Funcb_{t+1}^{\SDDP}$ and $x\in \X$.
 Using Equation~\eqref{Bellman:linear-convex} we obtain that
 $\B_t\left( \vopt_\Func \right) \np{x}$ is equal to $b$ the optimal value of optimization problem~\eqref{NodalProblemSDDP}.
 Thus, since $\func_t^{\SDDP}\left( \Func, x \right)\np{x}=b$ we obtain that the selection function is tight.
 It is also valid as $a$ is a subgradient of the convex function $\B_t\left( V_\Func \right)$ at $x$. For $t=T$, the selection
 function $\func_T^{\SDDP}$ is tight and valid by convexity of $V_T$.
\end{proof}

If we want to apply the convergence result from
Theorem~\ref{ConvergenceTheorem}, as we approximate from below the value
functions ($\optop = \sup$) then one has to make the draws according to some sets
$K_t^k$ such that the sets $K_t^* := \limsup_{k\in \N} K_t^k$ are $V_t^*$
optimal. As done in the literature of the Stochastic Dual Dynamic Programming
algorithm (see for example~\cite{Gi.Le.Ph2015} and~\cite{Zo.Ah.Su2018}
or~\cite{Pe.Pi1991}), one can study the case when the draws are made along the
optimal trajectories of the current approximations.

More precisely, fix $k \in \N$ we define a sequence $(x_0^k, x_1^k, \ldots , x_T^k)$ by
\begin{equation*}
 x_0^k := x_0
 \quad\text{and}\quad
 \forall t \in \ce{0,T-1}, \ x_{t+1}^k :=  f_t \np{ x_t^k, u_t^k}
 \eqsepv
\end{equation*}
where $u_t^k \in \argmin_u \B_t^u\left( V_t^k \right)\left(x_t^k\right)$. We say
that such a sequence $(x_0^k, x_1^k, \ldots , x_T^k)$ is an \emph{optimal
 trajectory for the $k$-th approximations starting from $x_0$}. We show that
optimal trajectories for the current approximations become $\np{V_t^*}$-optimal
when $k$ goes to infinity, using a result of convergence in minimization by
Rockafellar and Wets~\cite[Theorem 7.33]{Ro.We2009}.

\begin{proposition}
 \label{OptimalityOfTrajectories}
 For every $k\in \N$, let $(x_0^k, x_1^k, \ldots , x_T^k)$ be an optimal trajectory for the $k$-th approximations starting from $x_0$ and define a sequence of singletons for every $t\in \ce{0,T}$, $K_t^k := \left\{ x_t^k \right\}$.
 Then the sets $\np{K_t^*}_{t\in \ce{0,T}}$ defined by $K_t^* := \limsup_k K_t^k$ are $\np{V_t^*}$-optimal.
\end{proposition}

\begin{proof}
 Fix $t\in \ce{0,T{-}1}$, we want to show that Equation~\eqref{eq:optimalSets} is satisfied for
 $K_t^*$ which is equivalent to prove that for every $x^*_t \in K_t^*$, we have that
 \begin{equation}
  \label{eq:tmp1}
  \B_t\left( V_{t+1}^* + \indi{K_{t+1}^*} \right)\left(x^*_t\right) = \B_t\left( V_{t+1}^* \right)(x^*_t).
 \end{equation}
 Now, using the definition of the Bellman operators in Equation~\eqref{Bellman:linear-convex} and Equation~\eqref{eq:tmp1}  we have to prove that there exists a control $u^*_t \in U_t(x_t^*)$ such that
 \begin{equation}
  \label{eq:tmp3}
  u^*_t \in \bset{ u}{ f_t \np{x^*_t, u} \in K_{t+1}^*}
  \cap
  \argmin_{u\in U_t(x_t^*) } \Bp{ c_t\np{x_t^*, u} + V_{t+1}^* \bp{ f_t\np{x_t^*, u }}}\eqfinp
 \end{equation}

 Fix $x_t^* \in K_t^*$ and extracting if needed a subsequence, without loss of
 generality, assume that $\np{x_t^k}_{k\in \N}$ converges to $x_t^*$. Fix
 $k\in \N$ and the sequence of controls $\np{u_0^k, \ldots, u_{T-1}^k}$
 associated with the optimal trajectory for the $k$-th approximations
 $\np{x_0^k, \ldots, x_T^k}$. We have that
 \begin{equation}
  \label{eq:control-k-opt}
  u_t^k \in \bset{ u}{f_t\np{x_t^k, u} \in K_{t+1}^k} \cap
  \argmin_{u \in U_t\np{x_t^k}} \Bp{c_t\np{x_t^k, u} + V_{t+1}^k\bp{f_t\np{x_t^k, u}}}
  \eqfinp
 \end{equation}
 Extracting, if needed, a subsequence $\np{u_t^n}_{n\in \N}$ of
 $\np{u_t^k}_{n\in \N}$, we will show that the sequence $\np{u_t^n}_{n\in \N}$ converges to some
 $u_t^* \in \argmin_{u \in U_t(x_t^*)} c_t\np{x_t^*, u} + V_{t+1}^* \np{
   f_t\np{x_t^*,u}}$. Equation~\eqref{eq:tmp3} will be satisfied as for every $n\in \N$,
 $f_t\np{x_t^n, u_t^n} \in K_{t+1}^n$, the continuity of $f_t$ and definition of
 $K_{t+1}^*$ will ensure that
 $u_t^* \in \left\{ u \mid f_t\np{x_t^*, u} \in K_{t+1}^*\right\}$.

 We will use the result of convergence in minimization~\cite[Theorem 7.33]{Ro.We2009}. We define
 \begin{align*}
   & B^k : \U \to \overline{\R}, \ u \mapsto c_t\np{x_t^k, u} + V_{t+1}^k\bp{f_t\np{x_t^k,u}}\eqfinv \\
   & B^* : \U \to \overline{\R}, \ u \mapsto c_t\np{x_t^*, u} + V_{t+1}^*\bp{f_t\np{x_t^*,u}}\eqfinp
 \end{align*}
 Recall that, under Assumption~\ref{hypo:linear-convex}-\eqref{hypo:multiapplication}, the
 set-valued mapping $U_t$ has compact values with non-empty interior and is
 $L_{U_t}$-Lipschitz continuous for some constant $L_F > 0$. Moreover, the
 functions $B^*$ and every $B^k$, $k\in \N$ are convex, l.s.c., proper,
 inf-compact, with compact domains $U_t\np{x_t^*}$ and $U_t\np{x_t^k}$,
 respectively. As $U_t$ is Lipschitz continuous, the sequence of functions
 $\np{B^k}_{k\in \N}$ converges uniformly to $B^*$ on every compact $K$ included
 in the interior of $\dom\np{B^*} = U_t\np{x_t^*}$. Thus, by~\cite[Theorem
  7.17.c]{Ro.We2009}, $\np{B_t^k}_{k\in \N}$ epiconverges to $B^*$. Finally,
 $\np{u_t^k}_{k\in \N} \subset f_t\np{X_t, U_t\np{X_{t}}}$ which is compact as
 $U_t$ is u.s.c. and $f_t$ is continuous. We conclude that we can extract a
 converging subsequence out of $\np{u_t^k}_k$. Denoting by
 $u_t^* \in U_t\np{x_t^*}$ its limit, by~\cite[Theorem 7.33]{Ro.We2009} we
 finally have that
 \(
 u_t^* \in \argmin_{u \in \U} B^*\np{u}
 \). This ends the proof.
\end{proof}

Hence, when applying TDP with the SDDP selection function, we will refine the approximations along the current optimal trajectories, \emph{i.e.} we use Oracle defined in Example~\ref{example:opt_traj}. We conclude this section by proving the convergence of TDP algorithm in the SDDP case.

\begin{theorem}[Lower (outer) approximations of the value functions]
 Under Assumption~\ref{hypo:linear-convex}, for every $t\in \ce{0,T}$, denote by
 $\left( V_t^k \right)_{k\in \N}$ the sequence of functions generated by
 \NomAlgo \ with the selection function $\func_t^{\SDDP}$ and the draws
 made uniformly over the sets $K_t^k$ defined in
 Proposition~\ref{OptimalityOfTrajectories}.  Then, the sequence
 $\left( V_t^k \right)_{k\in \N}$ is non-decreasing, bounded from above by
 $V_t$, and converges uniformly to $V_t^*$ on every compact set included in
 $\dom\left(V_t \right)$. Moreover, almost surely over the draws, $V_t^* = V_t$
 on $\limsup_{k\in \N} K_t^k$.
\end{theorem}
\begin{proof}
 As the structural assumptions Assumption~\ref{Assumptions} are satisfied, as the functions
 $\func_t^{\SDDP}$, $0\leq t\leq T$, are compatible selections and the sets
 $\left(K_t^*\right)_{t\in \ce{0,T}}$ are $\np{V_t^*}$-optimal (case
 $\optop = \sup$) by Theorem~\ref{ConvergenceTheorem}, we have the result.
\end{proof}

\section{A min-plus selection function: upper approximations in the
  linear-quadratic framework with both continuous and discrete controls}
\label{sec:Exemples_switch}

In \S\ref{homogeneous_case}, we study the case where the cost functions and
dynamics are homogeneous.  We conclude this section in \S\ref{theExample} with an
example which shows that using optimal trajectories of the best current
approximations as trial points in \NomAlgo \ may generate functions
$\np{V_t^k}_{k\in \N}$ which do not converge to the value function $V_t$. In the appendix \ref{homogenization}, we show how one can
use the homogeneous case to solve the non-homogeneous case by augmenting the
state dimension by one.

\subsection{The pure homogeneous case}
\label{homogeneous_case} We will denote by $\Mat_n$ the set of $n{\times}n$ real matrices and by $\Sy_{n}\subset \Mat_n$ the subset of symmetric
matrices.

\begin{definition}[Pure quadratic form]
 \label{purequadform}
 We say that a function $q : \X \rightarrow \R$ is a pure quadratic form if there exist a symmetric real matrix $M \in \Sy_n$ such that for every $x\in \X$, we have
 \(
 q(x) = x^T M x.
 \)

 Similarly, a function $q : \X \times \U \rightarrow \R$ is a pure quadratic
 form if there exist two symmetric real matrices $M_1 \in \Sy_n$ and
 $M_2 \in \Sy_m$ such that for every $x \in \X$, we have
 \(
 q(x,u) = x^T M_1 x + u^T M_2 u.
 \)
\end{definition}

Let us insist that pure quadratic forms are not general $2$-homogeneous quadratic forms in the sense that they lack a mixing term of the form $x^T M u$. In \ref{homogenization} we show why we do not lose generality by studying this case instead of general polynomials of degree $2$. Let $\X = \R^n $ be a continuous state space (endowed with its euclidean and Borel structure), $\U = \R^m$ a continuous control space  and $\Discret$ a finite set of discrete (or switching) controls. We want to solve the following optimization problem
\begin{subequations}
 \label{pb:linear-quad-switch}
 \begin{align}
  \min_{\substack{ (x, u , \discret) \in \X^{T+1} \times \U^T \times \Discret^T}}
   & \sum_{t=0}^{T-1} c_t^{\discret_t} (x_t, u_t) + \psi(x_T)
  \\
  \text{s.t.}
   & \quad x_0 \in \X \ \text{given, and}\quad
  \forall t \in \ce{0,T-1}, \ x_{t+1} = f_t^{\discret_t}\np{ x_t, u_t}\eqfinp
 \end{align}
\end{subequations}
\begin{assumption}
 \label{hypo:linear-quad-switch}
 Let $t \in \ce{0,T-1}$ and $\discret \in \Discret$ be arbitrary.

 \noindent --
 The dynamic $f_t^\discret : \X \times \U \longrightarrow \X$ is linear. That is,
 \(
 f_t^\discret(x, u) = A_t^\discret x + B_t^\discret u,
 \)
 for some given matrices $A_t^\discret$ and $B_t^\discret$ of compatible dimensions.

 \noindent -- The cost function $c_t^\discret : \X \times \U \longrightarrow \R$ is a pure convex quadratic form,
 \(
 c_t^\discret(x,u) = x^T C_t^\discret x + u^TD_t^\discret u,
 \)
 where the matrix $C_t^\discret$ is symmetric semidefinite positive and the matrix $D_t^\discret$ is symmetric definite positive.

 \noindent -- The final cost function $\psi := \inf_{i\in I_T} \psi_i$ is a finite infimum of pure convex quadratic form, of matrix $M_i \in \Sy_n$ with $i\in I_T$ a finite set, such that there exists a constant $\alpha_T \geq 0$ satisfying for every $i \in I_T$
 \(
 \quad 0 \preceq M_i \preceq \alpha_T \Id
 \).
\end{assumption}

One can write the Dynamic Programming equation for Problem~\ref{pb:linear-quad-switch} as follows
\begin{equation}
 \label{DP:linear-quad-switch}
 V_T  = \psi
 \quad\text{ and }
 \forall t\in \ce{0,T{-}1}, \forall x\in \X,
 V_t\np{x} = \inf_{\discret\in \Discret}\inf_{u \in \U} c_t^\discret( x, u) + V_{t+1} \bp{f_t^\discret(x, u)}
 \eqfinp
\end{equation}

The following result is crucial in order to study this example: the value functions are $2$-homogeneous, allowing us to restrict their study to the unit sphere.
\begin{proposition}
 For every time step $t\in\ce{0,T}$, the value function $V_t$, solution of
 Equation~\eqref{DP:linear-quad-switch} is $2$-homogeneous, that is, for
 every $x\in \X$ and every $\lambda \in \R$, we have
 \(
 V_t\np{\lambda x} = \lambda^2 V_t\np{x}
 \).
\end{proposition}
\begin{proof}
 We proceed by backward recursion on time step $t\in \ce{0,T}$. For $t=T$ it is
 true by Assumption~\ref{hypo:linear-quad-switch}. Assume that it is true for some
 $t\in \ce{1,T}$. Fix $\lambda \in \R$, then by definition of $V_{t-1}$, for
 every $x\in \X$, we have
 \begin{align*}
  V_{t-1}\left(\lambda x \right)
   & = \min_{\discret \in \Discret}
  \min_{u\in U} c_{t-1}^\discret \np{\lambda x, u} +
  V_{t}\bp{ f_{t-1}^\discret\np{\lambda x,u}}
  \\
   & = \min_{\discret \in \Discret}
  \min_{u' = u/\lambda \in U} c_{t-1}^\discret\np{\lambda x, \lambda u'} +
  V_{t}\bp{ f_{t-1}^\discret\np{\lambda x,\lambda u'}}\eqfinv
 \end{align*}
 which yields the result by $2$-homogeneity of $x\mapsto c_{t-1}^\discret \np{x,u}$,
 linearity of $f_{t-1}^\discret$ and $2$-homogeneity of $V_{t}$.
\end{proof}

Thus, in order to compute $V_t$, one only needs to know its values on the unit
(euclidean) sphere $\Sphere$ as for every non-zero $x \in \X$,
$V_t(x) = \Vert x \Vert^2 \ V_t\bp{\frac{x}{\Vert x \Vert}}$. Hence, we will refine our approximations only on the sphere, that is we will draw trial points uniformly on the sphere and use the Oracle defined in Example~\ref{example:sphere}. Now, for every time
$t\in \ce{0,T{-}1}$ and every switching control $\discret \in \Discret$ we define
the \emph{Bellman operator with fixed switching control} $\B_t^\discret$ for
every function $\func : \X \rightarrow \overline{\R}$ by:
\[
 \B_t^\discret(\func ) := \inf_{u \in \U} c_t^\discret( \cdot, u) + \lVert f_t^\discret(\cdot, u) \rVert^2\func \left(\frac{f_t^\discret(\cdot, u)}{\lVert f_t^\discret(\cdot, u)\rVert} \right).
\]
For every time $t\in \ce{0,T-1}$ we define the \emph{Bellman operator} $\B_t$ for every function $\func : \X \to \overline{\R}$ by:
\begin{equation}
 \label{Bellman-min-plus}
 \B_t\left( \func \right) := \inf_{\discret \in \Discret} \B_t^\discret \left( \func \right).
\end{equation}
This definition of the Bellman operator emphasizes that the unit sphere
$\Sphere$ is $(V_t)$-optimal in the sense of Definition~\ref{optimalDraws}. Note that for
$2$-homogeneous functions, we have that
$\B_t^{\discret}\np{\func} = \inf_{u\in \U} c_t^{\discret}\np{\cdot,u} +
 \func\np{f_t^{\discret}\np{\cdot,u}}$. Using Equation~\eqref{Bellman-min-plus}, one can
rewrite the Dynamic Programming Equation~\eqref{DP:linear-quad-switch} as
\begin{equation}
 \label{DP:tmp}
 V_T  = \psi
 \quad\text{ and }\quad
 \forall t\in \ce{0,T{-}1}, V_t = \B_t\np{V_{t+1}}
 \eqfinp
\end{equation}
Now, in order to apply the \NomAlgo \ algorithm to Equation~\eqref{DP:tmp}, we
need to check Assumption~\ref{Assumptions}. Under Assumption~\ref{hypo:linear-quad-switch}, there
exist an interval in the cone of symmetric semidefinite matrices which is stable
by every Bellman operator $\B_t$ in the sense of the proposition below. We will
consider the Loewner order on the cone of (real) symmetric semidefinite
matrices, \emph{i.e.} for every couple of matrices of symmetric matrices
$\left(M_1, M_2\right)$ we say that $M_1 \preceq M_2$ if, and only if,
$M_2 - M_1$ is semidefinite positive. Moreover we will identify a pure quadratic
form with its symmetric matrix, thus when we write an infimum over symmetric
matrices, we mean the pointwise infimum over their associated pure quadratic
forms.

\begin{proposition}[Existence of a stable interval]
 \label{StableInterval}
 Under Assumption~\ref{hypo:linear-quad-switch}, we define a sequence of positive reals $\np{\alpha_t}_{t\in \ce{0,T}}$ by backward recursion on $t\in \ce{0,T-1}$ such that we have:
 \begin{equation}
  0 \preceq M \preceq \alpha_{t+1} \Id \Rightarrow 0 \preceq \B_t(M) \preceq \alpha_t \Id,
 \end{equation}
 where $\alpha_T$ is a given constant by Assumption~\ref{hypo:linear-quad-switch}.
\end{proposition}

\begin{proof}
 First, given an arbitrary $t\in \ce{0,T}$, we want to show that if
 $M \succeq 0$ then $\B_t (M) \succeq 0$. As in
 Proposition~\ref{propStructurale_SDDP} one can show that the Bellman operator $\B_t$ is
 order preserving. Therefore, if $M \succeq 0$ then
 $\B_t (M) \succeq \B_t (0)$. Hence it is enough to show that
 $\B_t (0) \succeq 0$. But by Formula~\eqref{equa:Riccati_reduced}, we have that
 $B_t (0) = \min_{\discret \in \Discret}C_t^\discret \succeq 0$ (by
 Assumption~\ref{hypo:linear-quad-switch}) hence the result follows.

 Second, let $t\in \ce{0,T-1}$ and $\alpha_{t+1} >0$ be fixed. We consider
 $\alpha_t > 0$ defined by
 \begin{equation}
  \label{alphadef}
  \alpha_t :=  \max_{\discret \in \Discret}  \alpha_{t+1} \lambdamax\bp{{A_t^\discret}\np{A_t^\discret}^T} + \lambdamax(C_t^\discret ) > 0
  \eqfinv
 \end{equation}
 and we prove that if $M\preceq \alpha_{t+1} \Id$ then we have that
 $\B_t (M) \preceq \alpha_t \Id$. For that purpose, consider $M$ such that
 $M\preceq \alpha_{t+1} \Id$. Then, denoting by $\overline{B}_t^\discret$ the
 matrix
 $\Id + \alpha_{t+1} B_t^\discret \np{D_t^\discret}^{-1} \np{B_t^\discret}^T$,
 we have that
 $\lambdamin(\overline{B}_t^\discret)= 1 + \alpha_{t+1}\lambdamin(B_t^\discret
  \np{D_t^\discret}^{-1} \np{B_t^\discret}^T)\ge 1$ using the fact that the
 matrix $B_t^\discret \np{D_t^\discret}^{-1} \np{B_t^\discret}^T$ is positive
 semi-definite by Assumptions~\ref{hypo:linear-quad-switch}. Now, we
 successively have for any $\discret \in \Discret$
 \begin{align}
  \lambdamax \bp{\B_t^\discret \np{M}}
   & \le \lambdamax \bp{\B_t^\discret \np{\alpha_{t+1}  \Id}}
  \tag{$\B_t^\discret$ is order preserving}
  \\
   & = \lambdamax
  \Bp{
   \alpha_{t+1} \np{A_t^\discret} ^T \np{\overline{B}_t^\discret}^{-1} A_t^\discret  + C_t^\discret
  }
  \tag{using~\eqref{equa:Riccati_reduced}}
  \\
   & \le
  \alpha_{t+1} \lambdamax
  \Bp{
   \np{A_t^\discret} ^T \np{\overline{B}_t^\discret}^{-1} A_t^\discret}  + \lambdamax\np{C_t^\discret}
  \tag{by Proposition~\ref{prop:valeurpropre}}
  \\
   & \le
  \alpha_{t+1}
  \lambdamax\bp{A_t^\discret {A_t^\discret} ^T}
  \lambdamax\bp{\np{\overline{B}_t^\discret}^{-1}}
  + \lambdamax\np{C_t^\discret}
  \tag{by Proposition~\ref{prop:valeurpropre}}
  \\
   & \le\alpha_{t+1}  \lambdamax\bp{A_t^\discret {A_t^\discret} ^T} + \lambdamax\np{C_t^\discret}
  \tag{as $\lambdamax\bp{\np{\overline{B}_t^\discret}}= \lambdamin\bp{\np{\overline{B}_t^\discret}}^{-1} \le 1$ }
  \\
   & \le \alpha_t \tag{using~\eqref{alphadef}}
  \eqfinv
 \end{align}
 which gives that $\B_t^\discret(M) \preceq \alpha_t \Id$. Then, the same result follows for the operator
 $\B_t$ using Equation~\eqref{Bellman-min-plus}. This ends the proof.
\end{proof}

Using Proposition~\ref{StableInterval},, one can deduce by backward recursion on $t\in \ce{0,T-1}$ the existence of intervals of matrices, in the Loewner order, which are stable by the Bellman operators.

\begin{corollary}
 \label{coro:StableBellman}
 Under Assumption~\ref{hypo:linear-quad-switch}, using the sequence of positive reals
 $\np{\alpha_t}_{t\in \ce{0,T}}$ defined in Proposition~\ref{StableInterval},, we define a
 sequence of positive reals $\np{\beta_t}_{t\in \ce{0,T}}$ by
 \( \beta_T := \alpha_T\) and \(\forall t\in \ce{0,T{-}1}\), $\beta_t:= \max\np{\alpha_t,\beta_{t+1}}$.
 Then, one has that
 \[
  0 \preceq M \preceq \beta_T \Id \Rightarrow \forall t \in \ce{0,T-1}, \ 0
  \preceq \B_{t}\np{\ldots \B_{T-2}\np{\B_{T-1}\np{M}}} \preceq \beta_t \Id .
 \]
\end{corollary}

The basic functions $\Funcb^{\MinPlus}_t$ will be pure quadratic convex forms bounded in the Loewner sense by $0$ and $\beta_t I$,
\[
 \Funcb_t^{\MinPlus} :=
 \bset{\func: x \in \X \mapsto x^TMx \in \R}{ M \in \Sy_n, \ 0\preceq M \preceq \beta_t \Id}\eqfinv
\]
and we define the following class of functions which will be stable by pointwise infimum of elements in $\Funcb_t^{\MinPlus}$,
\begin{equation}
 \label{def_funcbbQu}
 \Funcbb_t^{\MinPlus} :=
 \bset{\vopt_\Func}{\Func \subset \Funcb_t^{\MinPlus}}\eqfinp
\end{equation}

Exploiting the min additivity of the Bellman operator, which gives that
$\B_t\bp{\inf\np{\func_1, \func_2}} = \inf \bp{\B_t\np{\func_1},
  \B_t\np{\func_2}}$, and the fact that the final cost $\widetilde{\psi}$ is a
finite infima of basic functions, one deduces by backward induction on
$t\in \ce{0,T}$ that the value functions are finite infima of basic functions.

\begin{lemma}
 \label{FiniteInfimum}
 For every time $t\in \ce{0,T}$, there exists a finite set $\Func_t$ of convex pure quadratic forms such that
 \[
  V_t = \inf_{\func \in \Func_t} \func.
 \]
\end{lemma}
\begin{proof}
 For $t = T$, set $\Func_T := \left\{ \psi_i \right\}_{i\in I_T}$.  Now, assume
 that for some $t\in \ce{0,T-1}$, we have that
 \( V_{t+1} = \inf_{\func \in \Func_{t+1}} \func \), where $\Func_{t+1}$ is a
 finite set of convex pure quadratic functions. Then, by definition of the
 Bellman operators $\B_t$ (see Equation~\eqref{Bellman-min-plus}), we have that
 \begin{align*}
  V_t & = \B_t\np{V_{t+1}}
  =  \inf_{\discret \in \Discret} \B_t^{\discret}\np{\inf_{\func \in \Func_{t+1} } \func}
  = \inf_{\func \in \Func_{t+1} }\inf_{\discret \in \Discret}\B_t^{\discret} \np{\func}
  = \inf_{\func \in \Func_{t+1} , \discret \in \Discret} \bp{\B_t^{\discret} \np{\func}}
 \end{align*}
 Thus, setting $\Func_t := \ba{\B_t^{\discret} \np{\func}\,\vert\,\func \in \Func_{t+1} \text{ and } \discret \in \Discret}$,
 we obtain that $V_t = \inf_{\func \in \Func_t} \func$, where $F_t$  is a finite set of convex pure quadratic functions.
 Backward induction on time $t \in \ce{0,T}$ ends the proof.
\end{proof}

\begin{proposition}
 \label{struct_assump_lin-quad-switch}
 Under Assumption~\ref{hypo:linear-quad-switch}, the Problem~\ref{pb:linear-quad-switch}
 and the Bellman operators defined in Equation~\eqref{DP:linear-quad-switch} satisfy the
 structural assumptions given in Assumption~\ref{Assumptions}.
\end{proposition}

\begin{proof} We prove successively each assumption listed in Assumption~\ref{Assumptions}.

 \noindent $\bullet$ \ref{Assumptions}-\eqref{Stability-pointwise-optimum}.
 By construction, $\Funcbb_t^{\MinPlus}$  in Equation~\eqref{def_funcbbQu} is stable by pointwise infimum.

 \noindent $\bullet$ \ref{Assumptions}-\eqref{Stability-pointwise-convergence} and \ref{Assumptions}-\eqref{CommonRegularity}.
 We will show that every element of $\Funcbb_t^{\MinPlus}$ is $2\beta$-Lipschitz continuous on $\Sphere$. Let $\Func=\left\{ \func_i \right\}_{i \in I} \subset \Funcbb_t^{\MinPlus}$ with $I \subset \N$ and $\func_i \in \Funcb_t^{\MinPlus}$ with associated symmetric matrix $M_i$. Fix $x,y \in \Sphere$, we have successively
 \begin{align}
  \vert V_\Func(x) - V_\Func(y) \vert
   & =  \vert \inf_{i \in I} x^T M_i x - \inf_{i\in I} y^T M_i y  \vert                 \notag                    \\
   & \leq \max_{i \in I} \vert x^T M_i x - y^T M_i y  \vert                              \notag                   \\
   & \leq \max_{i \in I} \vert x^T M_i \left( x-y \right) + y^T M_i\left( x-y \right)  \vert        \notag        \\
   & \leq \max_{i \in I} \vert \langle x + y, M_i(x-y) \rangle \vert \tag{$M^T = M$}            \notag            \\
   & \leq  \Vert x + y \Vert \cdot \max_{i \in I} \Vert M_i\left(x-y\right) \Vert \tag{Cauchy-Schwarz}     \notag \\
   & \leq \Vert x + y \Vert \cdot \max_{i \in I} {\Vert M_i \Vert} \Vert x-y \Vert
  \notag
  \\
   & \leq \beta_t {\Vert x + y \Vert} \cdot  \Vert x-y \Vert
  \tag{$\Vert M_i \Vert \leq \beta_t$}
  \notag
  \\
   & \leq 2 \beta_t \Vert x - y \Vert \eqfinv
 \end{align}
 since $\Vert x + y \Vert \le 2$. Thus, every element of $\Funcbb_t^{\MinPlus}$ is $2\beta_t$-Lipschitz on
 $\Sphere$ and by stability by pointwise infimum, $\Funcbb_t^{\MinPlus}$ is
 stable by pointwise convergence.

 \noindent $\bullet$ \ref{Assumptions}-\eqref{Final-condition}.
 By Assumption~\ref{hypo:linear-quad-switch}, the final cost function $\psi$ is an element of $\Funcbb_T^{\MinPlus}$.

 \noindent $\bullet$ \ref{Assumptions}-\eqref{StabilityBellman}.
 This is given by Corollary~\ref{coro:StableBellman}.

 \noindent $\bullet$ \ref{Assumptions}-\eqref{order-preserving}.
 Proceed as in Proposition~\ref{propStructurale_SDDP}.

 \noindent $\bullet$ \ref{Assumptions}-\eqref{Additively-subhomogeneous}.
 Fix a time step $t\in \ce{0,T-1}$, a compact $K_t \subset \dom(V_t)\, (=\X)$, a function $\func \in \Funcbb_{t+1}^{\MinPlus}$ and a constant $\lambda \geq 0$. By definition of $\Funcbb_{t+1}^{\MinPlus}$, there exists a finite set $F := \left\{ \func_i \right\}_{i\in I} \subset \Funcb_{t+1}^{\MinPlus}$ such that $\func = \inf_{i \in I} \func_i$.
 By Equation~\eqref{equa:inverseRiccati_determinist}, for each $i\in I$ and $\discret \in \Discret$, there exists a linear map $L_i^{\discret}$ such that
 \begin{equation}
  \label{eq:tmp123}
  \min_{u\in \U} c_t^\discret\np{x,u} + \func_i\np{f_t^{\discret}\np{x,u}}= c_t^{\discret} \np{x,L_i^{\discret}(x)} + \func_i\np{f_t^{\discret}\np{x,L_i^{\discret}(x)}}
  \enspace ,
 \end{equation}
 with $\|L_i^{\discret}\|\leq \alpha_{t+1}C_t$,
 where $C_t$ is a constant depending on the parameters of the
 control problem only.
 Hence the maps $x\mapsto  f_t^{\discret}\np{x,L_i^{\discret}(x)}$
 are linear and their norm are bounded by $(\alpha_{t+1}+1)C'_t$
 for some constant $C'_t$ depending on the parameters of the control problem
 only. Set $M_t :=((\alpha_{t+1}+1)C'_t \|K_t\|)^2 $,
 where $ \|K_t\|$ is the radius of a ball centered in $0$ including $K_t$. Therefore, for $x\in K_t$, we have $\lVert f_t^{\discret}\np{x,u}\rVert^2 \leq M_t$. Now, for $x\in K_t$, using the bound on~$f_t$ we have
 \begin{align*}
  \B_t\left(\func + \lambda \right)\np{x}
   & = \min_{\substack{i\in I                                                                                                   \\ u\in \U \\ \discret \in \Discret}}
  c_t^\discret\np{x,u} + \lVert f_t^{\discret}\np{x,u}\rVert^2 \np{\func_i + \lambda}
  \np{ \frac{f_t^{\discret}\np{x,u} }{\lVert f_t^{\discret}\np{x,u}\rVert} }                                                    \\
   & \leq \min_{\substack{i\in I                                                                                                \\ \discret \in \Discret}}
  c_t^{\discret} \np{x,L_i^{\discret}(x)} + \lVert f_t^{\discret}\np{x,L_i^{\discret}(x)}\rVert^2
  \np{\func_i+\lambda} \np{ \frac{f_t^{\discret}\np{x,L_i^{\discret}(x)}}{\lVert f_t^{\discret}\np{x,L_i^{\discret}(x)}\rVert}} \\
   & \leq \min_{\substack{i\in I                                                                                                \\ \discret \in \Discret}}   c_t^{\discret} \np{x,L_i^{\discret}(x)} + \lVert f_t^{\discret}\np{x,L_i^{\discret}(x)}\rVert^2 \func_i
  \np{ \frac{ f_t^{\discret}\np{x,L_i^{\discret}(x)}}{\lVert f_t^{\discret}\np{x,L_i^{\discret}(x)}\rVert} } +M_t \lambda       \\
   & = \min_{\substack{i\in I                                                                                                   \\ \discret \in \Discret}}
  \Bp{c_t^{\discret} \np{x,L_i^{\discret}(x)} +  \func_i  \np{f_t^{\discret}\np{x,L_i^{\discret}(x)}}} +  M_t \lambda
  \\
   & = \min_{\substack{i\in I                                                                                                   \\ u\in \U \\ \discret \in \Discret}}
  \Bp{c_t^\discret\np{x,u} +  \func_i \bp{ f_t^{\discret}\np{x,u} }} + M_t \lambda \tag{by~\eqref{eq:tmp123}}                   \\
   & =  \B_t\np{\func}(x) +  M_t\lambda \eqfinp
 \end{align*}
 hence the desired result \(
 \B_t\left(\func + \lambda \right)\np{x}            \leq  M_t\lambda + \B_t\np{\func}(x)
 \).

 \noindent $\bullet$ \ref{Assumptions}-\eqref{proper-value}.
 This is a consequence of Lemma~\ref{FiniteInfimum}.

 \noindent $\bullet$ \ref{Assumptions}-\eqref{optimal-sets}
 Fix $\func \in \Funcbb_t$ and $\lambda \geq 0$. Denote by $\tilde{\func} = \func + \lambda$. For every $x\in \X$, we have that
 \begin{align*}
  \B_t\np{\tilde{\func}}\np{x} & = \min_{\np{u,v}\in \U \times \Discret} c_t^{\discret}\np{x,u} + \lVert f_t^{\discret}\np{x,u}\rVert^2 \tilde{\func}\np{\frac{f_t^{\discret}\np{x,u}}{\lVert f_t^{\discret}\np{x,u} \rVert}}                       \\
                               & = \min_{\np{u,v}\in \U \times \Discret} c_t^{\discret}\np{x,u} + \lVert f_t^{\discret}\np{x,u}\rVert^2 \np{\tilde{\func} + \indi{\Sphere}}\np{\frac{f_t^{\discret}\np{x,u}}{\lVert f_t^{\discret}\np{x,u} \rVert}} \\
                               & = \B_t\np{\tilde{\func} + \indi{\Sphere}}(x),
 \end{align*}
 which implies the desired result.

\end{proof}

\begin{remark}
 We have shown that $\B_t$ is additively subhomogeneous with constant $M_t$. An upper bound of $M_t$ can be computed as in the proof of Proposition~\ref{StableInterval}, by bounding the greatest eigenvalue of each matrices $L_i^{\discret}$.
\end{remark}

We now define, for any $t\in \ce{0,T}$, a selection functions $\func_t^{\MinPlus}$ and prove that it is a compatible selection function. As each $\argmin$ mentioned below involves a finite set, selecting an element in the $\argmin$ raises no issue.

\begin{proposition}
 \label{QuIsCompatible}
 For every time $t\in \ce{0,T}$, any $\Func \subset \Funcb_t^{\MinPlus}$ and any $x\in \X$,
 define a function $\func_t^{\MinPlus}$ as follows
 \begin{equation}
  \func_t^{\MinPlus}\np{\Func, x}
  \in
  \begin{cases}
   \B_t \Bp{ \argmin_{\func \in \Func} \bp{ \B_t \np{\func}\np{x}}} & \text{for}\quad t\not= T\eqsepv \\
   \argmin_{\psi_i \in \Func} \psi_i(x)                             & \text{for}\quad t=T
   \eqfinv
  \end{cases}
 \end{equation}
 is a \emph{compatible selection} function as defined in Definition~\ref{det:CompatibleSelection}.
\end{proposition}

\begin{proof}
 Fix $t=T$. The function $\func_t^{\MinPlus}$ is tight and valid as $V_T = \psi$. Now fix $t\in \ce{0,T-1}$. Let $\Func \subset \Funcbb_{t+1}^{\MinPlus}$ and $x\in \X$ be arbitrary. We have
 \begin{align*}
  \B_t\np{\vopt_\Func}(x)
   & =  \B_t \bp{ \inf_{\func \in \Func} \func} (x)
  \\
   & = \inf_{(u,\discret)\in \U \times \Discret}
  \Bp{c_t^\discret \np{x, u} + \inf_{\func \in \Func}\func \bp{f_t^\discret\np{x,u}}}
  \\
   & = \inf_{\func \in \Func}\inf_{(u,\discret)\in \U \times \Discret}
  \Bp{c_t^\discret \np{x, u}+  \func \bp{f_t^\discret\np{x,u}}}
  \\
   & = \inf_{\func \in \Func}\bp{ \B_t\left( \func \right)(x)}         \\
   & = \func_t^{\MinPlus}\left( \Func, x \right)\np{x}
  \eqfinp
 \end{align*}
 Thus, $\func_t^{\MinPlus}$ is tight.
 By similar arguments, we have for every $x' \in \X$ that
 \[
  \B_t\np{\vopt_\Func}(x') = \bp{\inf_{\func \in \Func} \B_t\np{\func}}\np{x'} \leq \func_t^{\MinPlus}\np{\Func, x}\np{x'}
  \eqfinp
 \]
 This shows that $\func_t^{\MinPlus}\np{\Func, x}$ is valid and ends the proof.
\end{proof}

We conclude this section by proving the convergence of TDP algorithm in the Min-plus case.
\begin{theorem}[Upper (inner) approximations of the value functions]
 For every $t\in \ce{0,T}$, denote by $\left( V_t^k \right)_{k\in \N}$ the
 sequence of functions generated by \NomAlgo \ with the selection function
 $\func_t^{\MinPlus}$ and the draws made uniformly over the sphere
 $K_t := \Sphere$. Under Assumption~\ref{hypo:linear-quad-switch}, the sequence
 $\left( V_t^k \right)_{k\in \N}$ is non increasing, bounded from below by
 $V_t$ and converges uniformly to $V_t^*$ on $\Sphere$. Moreover, almost surely
 over the draws, $V_t^* = V_t$ on $\Sphere$.
\end{theorem}
\begin{proof}
 As the structural Assumption~\ref{Assumptions} are satisfied, as the functions
 $\func_t^{\MinPlus}$, $0\leq t\leq T$ are compatible selections and the
 unit sphere $\Sphere$ is $V_t$-optimal (case $\optop = \inf$), we can apply Theorem \ref{ConvergenceTheorem}.
\end{proof}

\subsection{Optimal trajectories for upper approximations may not converge}
\label{theExample}
We now give an example showing that approximating from above may fail to converge when the
points are drawn along optimal trajectories for the current upper
approximations of $V_t$ (in contrast with Section~\ref{SDDP_Example} where we approximate from
below $V_t$). As shown by Proposition~\ref{HomogeneVSnonhomogene} there is no loss of
generality in considering the framework of \S\ref{homogeneous_case} but with
non-homogeneous functions.

\begin{figure}
 \centering
 \includegraphics[scale=0.7]{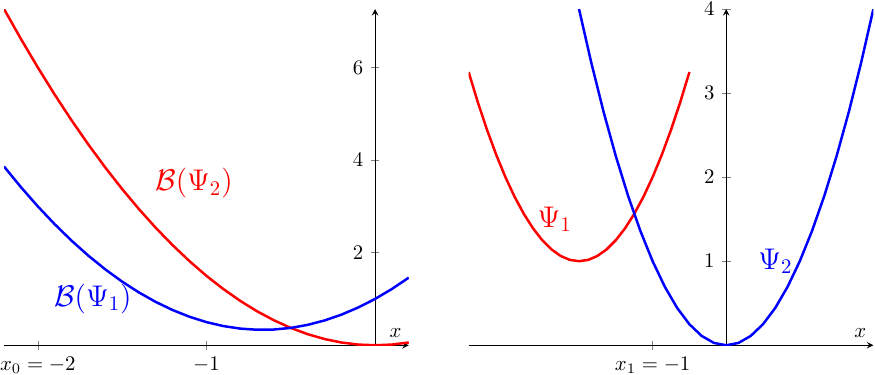}
 \caption{\label{ex:figure}Illustration of the multistage (two stages here)
  optimization problem studied in \S\ref{theExample}. The final cost function,
  $\psi = \inf\np{\psi_1, \psi_2}$ is shown in the right subfigure and
  displays of the images, by the Bellman operator $\B$, of the functions
  $\psi_1$ and $\psi_2$ are drawn in the left subfigure. We have that
  $\B(\psi_2)(x_0)> \B(\psi_1)(x_0)$. At the final time $t=1$, the ``best
  function'' at the point $-1$ is $\psi_2$. The image by the $k$-th optimal
  dynamic of $x_0 = -2$ is $x_1=-1$.}
\end{figure}

We consider a (non-homogeneous) problem with only two time steps, that is $T=1$ and $t \in \na{0,1}$ such that

\noindent $\bullet$ The state space $\X$ and the control space $\U$ are equal to $\R$.

\noindent $\bullet$  The linear dynamic is $f(x,u) = x + u$.

\noindent $\bullet$  The quadratic cost is $c(x,u) = x^2 + u^2$.

\noindent $\bullet$  The final cost function is the infimum, $\psi = \inf\np{\psi_1, \psi_2}$, of two given quadratic mappings,
$\psi_1(x)= (x+2)^2 +1$ and $\psi_2(x)= x^2$.

The Bellman operator $\B$, associated to this multistage optimization problem is
defined for every $\func : \X \to \overline{\R}$ and every $x\in \X$ by
\[
 \B\np{\func}(x) = \min_{u\in U}\bp{x^2 + u^2 + \func\np{x+u}} = x^2 + \min_{u\in U}\bp{u^2 + \func\np{x+u}}
 \eqfinp
\]
For the case where $\func_{a,b}\np{\cdot} = \np{\cdot+a}^2 + b$ with $a,b \in \R$ one has for every $x\in \R$
\begin{equation}
 \label{ex:image_bellman}
 \B\np{\func_{a,b}}\np{x} = \frac{3}{2}x^2 + ax + b.
\end{equation}
Fix $x_0 = x_0^k = -2$ for every $k\in \N$. As described in
Algorithm~\ref{\NomAlgo}, the approximations of the value functions $V_1$ and
$V_0$ are initialized to $+\infty$. Thus every control $u \in \U$ is optimal in
the sense that $u \in \argmin_{u' \in \U} x^2 + (u')^2 + \func\np{x+u'}$. Hence
if we set $x_1^0 := -1 = f(x_0, 1)$ then $(x_0,x_1^0)$ is an optimal trajectory
as described in Proposition~\ref{OptimalityOfTrajectories}.

We deduce from Equation~\eqref{ex:image_bellman} the following facts, illustrated in Figure~\ref{ex:figure}.
\begin{enumerate}
 \item The image of $\psi_2$ is strictly greater than the image of $\psi_1$ by
       the Bellman operator $\B$, \emph{i.e.}
       \[
        \B\np{\psi_2}(-2) > \B\np{\psi_1}\np{-2}.
       \]
 \item The image by the $k$-th optimal dynamic of $-2$ is $-1$, \emph{i.e.}
       setting $u_0^k := \argmin_{u' \in \U} (-2)^2 + (u')^2 + V_1^k \np{-2+u'}$
       (the $\argmin$ is here a singleton) one has
       \[
        f\np{-2, u_0^k} = -1.
       \]
 \item At the final step $t=1$, the best function at the point $-1$ is $\psi_2$, \emph{i.e.}
       \[
        \psi(-1) = \inf\np{\psi_1(-1), \psi_2(-2)} = \psi_2\np{-2}.
       \]
\end{enumerate}

From those three facts, one can deduce starting $x_0 = -2$ and $x_1= -1$, the
optimal trajectory for the current approximations will always be sent to
$x_1=-1$. But, as shown in the proof of Proposition~\ref{QuIsCompatible} one can
show that the image by $\B$ of an infimum is the infimum of the images by $\B$:
\[
 V_0(-2) = \B\np{ \inf \np{\psi_1, \psi_2} }(-2) = \inf\np{  \B\np{  \psi_1 }(-2) , \B \np{ \psi_1 }(-2)}.
\]
Thus for every $k \in \N$, we have
\(
V_0(-2) =  \B\np{  \psi_1 }(-2) < \B \np{ \psi_1 }(-2) = V_0^k\np{-2}
\). The constant sequence $V_0^k\np{-2}$ fails to converge to $V_0\np{-2}$.

\section{Numerical experiments on a toy example}
\label{sec:ToyExample}

In \S\ref{constrained-linquad}, we propose a toy optimization Problem~\eqref{ConstrainedLQ}  on which we run TDP-SDDP and TDP-Minplus. Problem~\eqref{ConstrainedLQ} falls in the framework described in Section~\ref{SDDP_Example}. Thus, we are able to obtain lower approximations of $V_t$ using TDP-SDDP. TDP-Minplus cannot be applied directly. We apply a ``discretization'' step to Problem~\eqref{ConstrainedLQ} (see \S\ref{Discretization}) which yields Problem~\eqref{SwitchedConstrainedLQ} parameterized by an integer $N > 0$. Then we apply to Problem~\eqref{SwitchedConstrainedLQ} an ``homogenization'' step (see \S\ref{Homogeneization}) to obtain Problem~\eqref{HomogeneizedSwitchedConstrainedLQ}. The value functions $V_t$ of the original Problem~\eqref{ConstrainedLQ} are bounded from above by $\vtilden$, the value functions of Problem~\eqref{HomogeneizedSwitchedConstrainedLQ}. We apply TDP-Minplus (described in Section~\ref{sec:Exemples_switch}) to Problem~\eqref{HomogeneizedSwitchedConstrainedLQ} which gives upper approximations of $\vtilden$ and \emph{a fortiori}, of $V_t$. In \S\ref{numerical_experiments}, we show numerical experiments which show the convergence of this approximation scheme to $V_t$.

\subsection{A toy example: constrained linear-quadratic framework}
\label{constrained-linquad}
Let $\vmin,\vmax$ be two given reals such that $\vmin < \vmax$, we study the
following multistage linear quadratic problem involving a constraint on one of
the controls:
\begin{subequations}
 \label{ConstrainedLQ}
 \begin{align}
  \min_{(x,u,v) \in \X^{T+1} \times \U^{T} \times [\vmin,\vmax]^{T} }
   & \sum_{t=0}^{T-1} c_t(x_t, u_t,v_t) + \psi(x_T)
  \\
  \text{s.t.}
   & \quad x_0 \in \X \ \text{given, and}\quad
  \forall t \in \ce{0,T-1}, \ x_{t+1} = f_t(x_t, u_t, v_t)   \eqfinv
 \end{align}
\end{subequations}
where $\X=\R^n$ and $\U=\R^m$, with quadratic convex costs functions of the
form \[c_t(x,u,v) = x^T C_t x + u^T D_t u + v^2 d_t,\] where $C_t \in\Sn_n$,
$D_t \in \Snp_n$ and $d_t>0$, linear dynamics
$f_t(x,u, v) = A_t x + B_t u + v b_t$, where $A_t$ (resp. $B_t$) is a
$n \times n$ (resp. $n \times m$) matrix, $b_t \in \X$, and final cost function
$\psi := x^T M x$ with $M \in \Snp_n$.


For every $t\in \ce{0,T}$, the value function $V_t$ is $L_t$-Lipschitz
continuous and convex. Moreover the Lipschitz constant $L_t > 0$ can be
explicitly computed. As done in Section~\ref{SDDP_Example} we will generate
lower approximations of the value functions $V_t$ through compatible selection
functions $\np{\func_t^{\SDDP}}_{t\in \ce{0,T}}$. In this example, the
structural Assumption~\ref{Assumptions} are not satisfied as the sets of states
and controls are not compacts. As we will still observe convergence of the lower
approximations generated by \nomalgo
$\np{\func_t^{\SDDP}}_{t\in \ce{0,T}}$ to the value functions, this
suggests that the (classical) framework presented in Section~\ref{SDDP_Example}
can be extended. This will be the object of a future work and here we would like
to stress on the numerical scheme and results.

\subsection{Discretization of the constrained control}
\label{Discretization}

We approximate Problem~\eqref{ConstrainedLQ} by discretizing the constrained
control in order to obtain an unconstrained switched multistage linear
quadra\-tic problem. More precisely, we fix an integer $N \geq 2$, set
$ v_i = \vmin + i\frac{\vmax-\vmin}{N-1}$ for every $i \in \ce{0,N-1}$ and set
$\V:= \left\{v_0, v_1, \ldots v_{N-1} \right\}$.  Then, we define the following
unconstrained switched multistage linear quadratic problem:
\begin{subequations}
 \label{SwitchedConstrainedLQ}
 \begin{align}
  \min_{(x,u,v) \in \X^{T+1} \times \U^{T} \times \V^{T} }
   & \sum_{t=0}^{T-1} c_t^{v_t} (x_t, u_t) + \psi(x_T)
  \\
  \text{s.t.}
   & \quad x_0 \in \X \ \text{given, and}\quad
  \forall t \in \ce{0,T-1}, \ x_{t+1} = f_t^{v_t}(x_t, u_t)\eqfinv
 \end{align}
\end{subequations}
where for every $v\in \V$, $f_t^v = f_t\np{\cdot, \cdot, v}$ and
$c_t^{v} = c_t\np{\cdot, \cdot, v}$.  As the set of controls of Problem
\eqref{ConstrainedLQ} contains the set of controls of Problem
\eqref{SwitchedConstrainedLQ}, upper approximations of the value functions of
Problem \eqref{SwitchedConstrainedLQ} will are give upper approximations of the
value functions of Problem \eqref{ConstrainedLQ}. Thus we will construct upper
approximations for Problem~\eqref{SwitchedConstrainedLQ}.

\subsection{Homogenization of the costs and dynamics}
\label{Homogeneization}
We add a dimension to the state space in order to homogenize the costs and
dynamics, when a sequence of switching controls is fixed. Define the following
homogenized costs and dynamics for every $t\in \ce{0,T-1}$ by:
\begin{align*}
 \tilde{f_t}^v\np{x,y,u} & = \begin{pmatrix} A_t & vb_t \\ 0 & 1 \end{pmatrix} \begin{pmatrix} x \\ y \end{pmatrix} + \begin{pmatrix} B_t \\ 0 \end{pmatrix} u,             \\[1ex]
 \tilde{c_t}^v (x,y,u)   & = \begin{pmatrix} x \\ y \end{pmatrix}^T \begin{pmatrix} C_t & 0 \\ 0 & v^2 d_t \end{pmatrix}  \begin{pmatrix} x \\ y \end{pmatrix} + u ^T D_t u,
\end{align*}
And as the final cost function is already homogeneous,
$\widetilde{\psi}(x,y) = \begin{pmatrix} x \\ y \end{pmatrix}^T \begin{pmatrix}
  M & 0 \\ 0 & 0 \end{pmatrix} \begin{pmatrix} x \\ y \end{pmatrix}$. Using
these homogenized functions we define a multistage optimization problem with one
more (compared to Problem~\eqref{SwitchedConstrainedLQ}) dimension on the state
variable:
\begin{equation}
 \label{HomogeneizedSwitchedConstrainedLQ}
 \begin{aligned}
   & \min_{{(x,y,u,v) \in \X^{T+1} \times \R^{T+1} \times \U^T \times \V^{T}}}  \sum_{t=0}^{T-1} \tilde{c_t}^{v_t} (x_t, y_t, u_t) + \widetilde{\psi}(x_T, y_T) \\
   & \text{s.t.}  \   \left\{
  \begin{aligned}
    & (x_0, y_0) \in \X \times \R \ \text{is given,}                                           \\
    & \forall t \in \ce{0,T-1}, \ (x_{t+1},y_{t+1}) = \tilde{f_t}^{v_t}(x_t, y_t, u_t) \eqfinp \\
  \end{aligned}\right.
 \end{aligned}
\end{equation}

One can deduce the value functions $V_{t,N}$ of the multistage optimization
problem \eqref{SwitchedConstrainedLQ} (with non-homogeneous costs and dynamics)
from the value functions $\vtilden$ of~\eqref{HomogeneizedSwitchedConstrainedLQ}
(with homogeneous costs and dynamics) by
Proposition~\ref{HomogeneVSnonhomogene}. For every $x\in \X$, we have that
\begin{equation}
 \label{Homogeneisation}
 V_{t,N}(x) = \vtilden(x,1) \eqfinp
\end{equation}

For every time step $t\in\ce{0,T}$ the value function $\vtilden$ solution of
Problem~\eqref{HomogeneizedSwitchedConstrainedLQ} is $2$-homogeneous. That is,
for every $(x,y)\in \X\times \R$ and every $\lambda \in \R$, we have
$ \vtilden\left(\lambda x, \lambda y \right) = \lambda^2 \vtilden \left( x, y
 \right).$ This will allow us to restrict the study of the value functions to the
unit sphere, which is compact.

\subsection{Min-plus upper approximations of the value functions of Problem~\eqref{HomogeneizedSwitchedConstrainedLQ}}
\label{qu-sec}

We apply the results of Section~\ref{sec:Exemples_switch} as follows.
Let $v\in \V$ be a given switching control, in this framework, the operator $\B_t^v$ is defined as in Section~\ref{sec:Exemples_switch} but with an augmented state. More precisely, for every function $\func : \X \times \R \to \overline{\R}$ and every point $(x,y) \in \X \times \R$:
\begin{align*}
 \B_t^v \left( \func \right)(x,y) = \inf_{u\in \U} \tilde{c}_t^v(x,y, u) + \lVert \tilde{f}_t^v(x, y, u) \rVert^2\func\np{\frac{\tilde{f}_t^v(x, y, u)}{\lVert \tilde{f}_t^v(x, y, u) \rVert}} \eqfinp
\end{align*}
Then,  for every time $t\in \ce{0,T-1}$,
the Dynamic Programming operator $\B_t $ associated to Problem~\eqref{HomogeneizedSwitchedConstrainedLQ}
satisfies $\B_t \left( \func \right) := \inf_{v\in \V} \B_t^v  \left( \func \right)$.

A key property of the operators $\B_t^v$ and $\B_t$ is that they are
min-additive, meaning that for every functions
$\func_1, \func_2 : \X \to \overline{\R}$ one has:
\[
 \B_t^v \bp{ \inf\np{\func_1, \func_2} } = \inf\bp{\B_t^v\np{\func_1}, \B_t^v\np{\func_2}}\enspace ,
\]
and a similar equation for $\B_t$. Moreover, by Riccati formula (see
Equation~\eqref{equa:Riccati_determinist}), the image of a convex quadratic
function by $\B_t^v$ is also a convex quadratic function.

Lemma~\ref{FiniteInfimum} suggests
to use the following set of basic functions:
\[
 \Funcb_t^{\MinPlus} := \Func_t \ \text{ and } \ \Funcbb_t^{\MinPlus} := \left\{ \vopt_\Func \ \Big\vert \ \Func \subset \Funcb_t^{\MinPlus} \right\} \eqfinp
\]
As done in Section~\ref{sec:Exemples_switch}, one could also have considered as
basic functions the quadratic functions bounded in the Loewner sense between $0$
and $\alpha_t \id$, where $\alpha_t > 0$, $t \in \ce{0,T}$, are real numbers
such that, if $\func$ is a quadratic form bounded between $0$ and
$\alpha_{t+1}\id$, then $\B_t^v\np{\func}$ is bounded between $0$ and
$\alpha_t\id$.

Moreover, using the aforementioned properties, one will be able to compute
$\B_t^v\np{ \vopt_\Func}$ for a given switching control $v$ and
$\B_t\np{ \vopt_\Func}$, for any finite set $F$ of convex quadratic functions.
Therefore, given a time $t\in \ce{0,T-1}$, we define the selection function
$\func_t^{\MinPlus}$ as follows.  For any given
$\Func \subset \Funcb_{t+1}^{\MinPlus}$ and $(x,y)\in \X\times \R$,
\begin{align*}
  & \func_t^{\MinPlus}\left(\Func, x, y \right) =\B_t^v(\func)                                                               \\
  & \text{for some}\quad (v,\func)\in \argmin_{(v,\func)\in \V\times \Func} \B_t^v \left( \func \right) \left( x, y \right).
\end{align*}
Moreover, at time $t=T$, for any $\Func \subset \Funcb_T^{\MinPlus}$ and $(x,y) \in \X \times \R$, we set
\[
 \func_t^{\MinPlus}\left(\Func, x,y \right) = \widetilde{\psi}(x,y) = \psi(x).
\]

Motivated by the 2-homogeneity of the value functions, the random draws of TDP
for the basic functions $\Funcb_t^{\MinPlus}$, $1 \leq t \leq T$ and the
selection functions $\func_t^{\MinPlus}$ will be made uniformly on the
unit euclidean sphere, which satisfies
Definition~\ref{AssumptionOracle}. Indeed, by $2$-homogeneity, it is enough to
know the value functions of \eqref{HomogeneizedSwitchedConstrainedLQ} on the
sphere to know them on the whole state space.

\subsection{Upper and lower approximations of the value functions}

For a large number of discretization points $N$ (defined in
\S\ref{Discretization}), one would expect that the value functions $V_{t,N}$
of~\eqref{SwitchedConstrainedLQ} approximate the value functions $V_t$
of~\eqref{ConstrainedLQ}.
Indeed, one can show that for every time step $t\in \ce{0,T}$, the approximation
error is bounded by $C_t T/N^2$ in $\X$, for some constant $C_t > 0$.
%
Thus, for large $N$, we have $V_{t,N} \approx V_t$ and by Equation~\eqref{Homogeneisation}, for every $N \geq 2$, we have
\[
 \vtilden\np{\cdot, 1} = V_{t,N}\geq V_t.
\]
In the following Proposition we approximate $\vtilden$ from above by a min-plus
algorithm and $V_t$ from below by SDDP
and using the convergence result of Theorem~\ref{ConvergenceTheorem} (admitting that the result still holds for SDDP in this framework), we obtain the following one.
%
\begin{theorem}
 \label{ConvergenceTheoremExample}
 For every $t\in \ce{0,T}$, denote by
 $\left( \overline{V}_t^k \right)_{k\in \N}$ (resp.
 $\np{ \underline{V}_t^k}_{k\in \N}$) the sequence of functions generated by TDP
 with the selection function $\func_t^{\MinPlus}$
 (resp. $\func_t^{\SDDP}$) and the draws made uniformly over the euclidean
 sphere of $\X\times \R$ (resp. made as described in
 Proposition~\ref{OptimalityOfTrajectories}).

 Then the sequence $\left( \overline{V}_t^k \right)_{k\in \N}$
 (resp. $\np{\underline{V}_t^k}_{k \in \N}$) is non-increasing
 (resp. non-decreasing), bounded from below (resp. above) by $\vtilden$
 (resp. $V_t$) and converges uniformly to $\vtilden$ (resp. $V_t$) on any
 compact subset of $\X\times \R$ (resp. $K_t^*$ defined in
 Proposition~\ref{OptimalityOfTrajectories}).
\end{theorem}

\subsection{Numerical experiments}
\label{numerical_experiments}


The following data was used as a specific case of \eqref{ConstrainedLQ}. For every time $t\in \ce{0,T-1}$,
\begin{align*}
 A_t & = (1-0.1)\Id \quad & B_t & = \begin{pmatrix} 1 & \cdots & 1 \\ \vdots &  & \vdots \\ 1 & \cdots & 1 \end{pmatrix} \quad & b_t & = \begin{pmatrix} 1 \\ \vdots \\ 1\end{pmatrix} \\
 C_t & = 0.1 \Id \quad    & D_t & = 0.1 \Id \quad                     & d_t & = 0.1.
\end{align*}
The time horizon is $T = 15$, the states are in $\X = \R^n$ with $n = 25$, the
unconstrained continuous controls are in $\U = \R^m$ with $m = 3$, the
constrained continuous control is in $[\vmin, \vmax]$,
with 
$[\vmin,\vmax]=[1,5]$ in the first example and $[\vmin,\vmax]=[-3,5]$ in the
second one.  Moreover, we start from the initial point
$x_0 = 0.2 \; (1 , \ldots , 1 )^T$ when TDP is applied with the selection
function $\func_t^{\SDDP}$ and the number of discretization points $N$ is
varying from $5$ to $200$, for TDP with the selection function
$\func_t^{\MinPlus}$.  In Figures~\ref{FirstExample}
and~\ref{SecondExample}, we give graphs representing the values
$\underline{V}_t^k\np{x_t^k}$ and $\overline{V}_t^k\np{x_t^k,1}$ at each time
step $t\in \ce{0,T-1}$ where the sequence of states $\np{x_t^k}_{k\in \N}$ is
the optimal trajectory for the current lower approximations
$\np{\underline{V}_t^k}_{k\in \N}$ defined in \eqref{OptimalityOfTrajectories}.
From Theorem~\ref{ConvergenceTheoremExample}, we know that for every
$t \in \ce{0,T-1}$ the gap
$\overline{V}_t^k\np{x_t^k,1} - \underline{V}_t^k\np{x_t^k}$ should be close to
$0$ as $k$ increases assuming that $N$ is large enough to have
$V_t \approx V_{t,N}$.

\begin{figure*}
 \centering
 \includegraphics[scale=1]{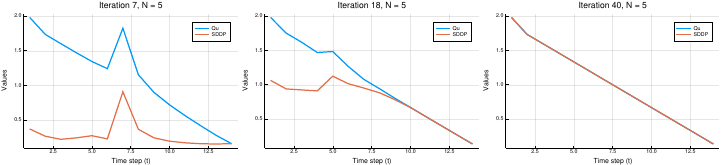}
 \caption{\label{FirstExample} First example for $\vmin=1$, $\vmax=5$ with
  $N = 5$ after $7$ iterations (left), $18$ iterations (middle) and $40$
  iterations (right).}
\end{figure*}

\begin{figure*}
 \centering
 \includegraphics[scale=1]{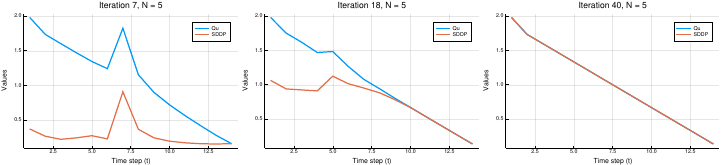}
 \caption{Second example for $\vmin=-3$, $\vmax=5$ with varying $N = 5$ (left),
  $N = 50$ (middle) and $N = 200$ (right) after $20$ iterations.}
 \label{SecondExample}
\end{figure*}

\begin{figure*}
 \centering
 \includegraphics[scale=1]{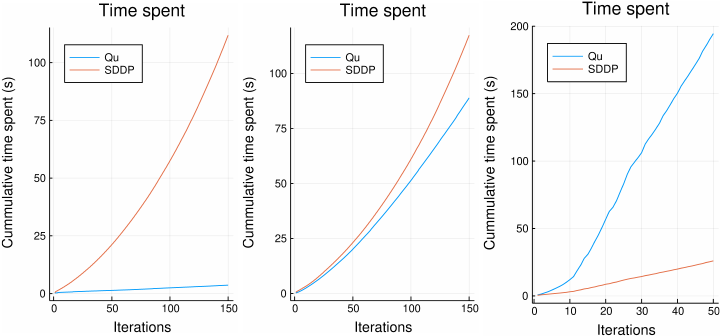}
 \caption{\label{TimePlots} Time spent for the first example (left) and the
  second example when $N=50$ (middle) and $N=200$ (right).}
\end{figure*}

On those two examples, we exhibit two convergence behaviors. On the first
example, the constrained control has to be greater than $1$, thus avoiding $0$
which would have been (or almost) the optimal control if there were no
constraint. The optimal constrained control is the projection on
$\U \times [\vmin,\vmax]$ of the optimal unconstrained control, thus the
switching control is most of the time equal to the lower bound $\vmin=1$.

From this observation we deduce two properties. First, the upper approximation
given by Qu algorithm is good, even for a small $N$, as the optimal switch is
(most of the time) equal to $\vmin$. Second, this implies that at iteration $k$,
the set $\Func_t^k$ is of small cardinality.

Moreover, in this example the number of switches is $N = 5$ thus few
computations of $\B_t^v\np{\func}(x)$ need to be done in order to compute
$\B_t\np{\func}(x)$ for some $x \in \X$ and $\func \in \Func_t^k$.  Thus, as
shown on the left of Figure~\ref{TimePlots}, the computation time of an
iteration of Qu's min-plus algorithm is small compared to SDDP which does not
exploit this property.

On the second example, the constrained control is in an interval containing
$0$. The switching control often changes and this means more computations. A
compromise between computational time and precision can be achieved (see
Figure~\ref{TimePlots}) in order to make the computational time of Qu algorithm
similar to the one of SDDP algorithm.

\section*{Conclusion}
In this article we have devised an algorithm, \NomAlgo, that encompasses both a
discrete time version of Qu's min-plus algorithm and the SDDP algorithm in the
deterministic case. We have shown in the last section that \NomAlgo \ can be
applied to two natural frameworks: one for min-plus and one for SDDP. In the
case where both framework intersects, one could apply \NomAlgo \ with the
selection functions $\func_t^{\MinPlus}$ and get non-increasing upper
approximations of the value function. Simultaneously, by applying \NomAlgo \
with the selection function $\func_t^{\SDDP}$, one would get
non-decreasing lower approximations of the value function. Moreover, we have
shown that the upper approximations are, almost surely, asymptotically equal to
the value function on the whole space of states $\X$ and that the lower
approximations are, almost surely, asymptotically equal to the value function on
a set of interest.

Thus, in those particular cases we get converging bounds for $V_0(x_0)$, which
is the value of the multistage optimization Problem~\ref{MultistageProblem},
along with asymptotically exact minimizing policies. In those cases, we have
shown a possible way to address the issue of computing efficient upper bounds
when running the SDDP algorithm by running in parallel another algorithm (namely
\nomalgo \ with min-plus selection functions).

In Section~\ref{sec:ToyExample} we studied a way to simultaneously build lower
and upper approximations of the value functions using the results of the
previous sections. However the discretization and homogenization scheme that
was described is rapidly limited by the dimension of the control space, due to
the need to discretize the constrained controls. We will provide in a future
work, a systematic scheme to use simultaneously SDDP and a min-plus methods
which is more efficient numerically and does not rely on discretization of the
control space. Moreover we will extend the current framework to multistage
stochastic optimization problems with finite white noises.

\bibliographystyle{plain}

\begin{thebibliography}{10}

\bibitem{Au.Ek1984}
Jean~Pierre Aubin and Ivar Ekeland.
\newblock {\em Applied Nonlinear Analysis: {{Jean}}-{{Pierre Aubin}} and {{Ivar
  Ekeland}}}.
\newblock Pure and Applied Mathematics: {{A Wiley}}-{{Interscience}} Series of
  Texts, Monographs, and Tracts. {Wiley}, {New York}, 1984.

\bibitem{Be1954}
Richard Bellman.
\newblock The theory of dynamic programming.
\newblock {\em Bulletin of the American Mathematical Society}, 60(6):503--515,
  November 1954.

\bibitem{Be2016}
Dimitri~P. Bertsekas.
\newblock {\em Dynamic Programming and Optimal Control}, volume~1 of {\em
  Athena Scientific Optimization and Computation Series}.
\newblock {Athena Scientific}, {Belmont, Mass}, fourth edition, 2016.

\bibitem{Ci1989}
Philippe~G. Ciarlet.
\newblock {\em Introduction to {{Numerical Linear Algebra}} and
  {{Optimisation}}:}.
\newblock {Cambridge University Press}, first edition, August 1989.

\bibitem{Dr2002}
Stuart Dreyfus.
\newblock Richard {{Bellman}} on the {{Birth}} of {{Dynamic Programming}}.
\newblock {\em Operations Research}, 50(1):48--51, February 2002.

\bibitem{Gi.Le.Ph2015}
P.~Girardeau, V.~Leclere, and A.~B. Philpott.
\newblock On the {{Convergence}} of {{Decomposition Methods}} for {{Multistage
  Stochastic Convex Programs}}.
\newblock {\em Mathematics of Operations Research}, 40(1):130--145, February
  2015.

\bibitem{Gu2014}
Vincent Guigues.
\newblock {{SDDP}} for some interstage dependent risk-averse problems and
  application to hydro-thermal planning.
\newblock {\em Computational Optimization and Applications}, 57(1):167--203,
  January 2014.

\bibitem{Gu.Ro2012}
Vincent Guigues and Werner R{\"o}misch.
\newblock Sampling-{{Based Decomposition Methods}} for {{Multistage Stochastic
  Programs Based}} on {{Extended Polyhedral Risk Measures}}.
\newblock {\em SIAM Journal on Optimization}, 22(2):286--312, January 2012.

\bibitem{La.Ro1995}
Peter Lancaster and L.~Rodman.
\newblock {\em Algebraic {{Riccati}} Equations}.
\newblock Oxford Science Publications. {Oxford University Press}, 1995.

\bibitem{Mc2007}
William~M. McEneaney.
\newblock A {{Curse}}-of-{{Dimensionality}}-{{Free Numerical Method}} for
  {{Solution}} of {{Certain HJB PDEs}}.
\newblock {\em SIAM Journal on Control and Optimization}, 46(4):1239--1276,
  January 2007.

\bibitem{Pe.Pi1991}
M.~V.~F. Pereira and L.~M. V.~G. Pinto.
\newblock Multi-stage stochastic optimization applied to energy planning.
\newblock {\em Mathematical Programming}, 52(1-3):359--375, May 1991.

\bibitem{Qu2013}
Zheng Qu.
\newblock {\em Nonlinear {{Perron}}-{{Frobenius}} Theory and Max-plus Numerical
  Methods for {{Hamilton}}-{{Jacobi}} Equations}.
\newblock PhD thesis, Ecole Polytechnique X, October 2013.

\bibitem{Qu2014}
Zheng Qu.
\newblock A max-plus based randomized algorithm for solving a class of {{HJB
  PDEs}}.
\newblock In {\em 53rd {{IEEE Conference}} on {{Decision}} and {{Control}}},
  pages 1575--1580, December 2014.

\bibitem{Ro.We2009}
Ralph~Tyrrell Rockafellar and Roger J.-B. Wets.
\newblock {\em Variational Analysis}.
\newblock Number 317 in Die {{Grundlehren}} Der Mathematischen
  {{Wissenschaften}} in {{Einzeldarstellungen}}. {Springer}, {Dordrecht}, corr.
  3. print edition, 2009.

\bibitem{Sc1995}
Laurent Schwartz.
\newblock {\em Analyse. 1: {{Th\'eorie}} Des Ensembles et Topologie}.
\newblock Number~42 in Collection {{Enseignement}} Des Sciences. {Hermann},
  {Paris}, nouv. tirage edition, 1995.

\bibitem{Sh2011}
Alexander Shapiro.
\newblock Analysis of stochastic dual dynamic programming method.
\newblock {\em European Journal of Operational Research}, 209(1):63--72,
  February 2011.

\bibitem{Zo.Ah.Su2018}
Jikai Zou, Shabbir Ahmed, and Xu~Andy Sun.
\newblock Stochastic dual dynamic integer programming.
\newblock {\em Mathematical Programming}, March 2018.

\end{thebibliography}

\appendix

\section{Algebraic Riccati Equation}

This section gives complementary results for Section~\ref{sec:Exemples_switch}. We use the same framework and notations introduced in Section~\ref{sec:Exemples_switch}.

\begin{proposition}
 Fix a discrete control $\discret \in \Discret$ and a time step $t\in \ce{0,T-1}$.
 \begin{myenumerate}
  \item The operator $\B_t^\discret: \Sy_n\to \Sy_n^+$ restricted to the pure
  quadratic forms (identified with $\Sy_n$ the space of the symmetric
  semidefinite positive matrices) is given by the \emph{discrete time algebraic Riccati equation}
  \begin{align}
   \label{equa:Riccati_determinist}
   \B_t^\discret\np{M} = C_t^\discret + \np{A_t^\discret}^T M A_t^\discret -
   \np{A_t^\discret}^T M B_t^\discret
   \bp{D_t^\discret + B_t^\discret  M\np{B_t^\discret}^T }^{-1}
   \np{B_t^\discret}^{T} M A_t^\discret
   \eqfinp
  \end{align}
  \item Moreover, when $M \in \Sy_n^+$ Equation~\eqref{equa:Riccati_determinist} can be rewritten as
  \begin{equation} \label{equa:Riccati_reduced}
   \B_t^\discret(M) = \left(A_t^\discret\right)^TM\left(I + B_t^\discret\left(D_t^\discret\right)^{-1}\left(B_t^\discret\right)^TM\right)^{-1}A_t^\discret + C_t^\discret.
  \end{equation}
 \end{myenumerate}
\end{proposition}

\begin{proof}\quad

 \noindent $\bullet$ We prove Equation~\eqref{equa:Riccati_determinist}. Note that if $M$ is symmetric, then $\B_t^\discret(M)$ is also symmetric.
 Now, let $t \in \na{T{-}1, T{-}2, \ldots, 0}$ and $M \in \Sy_n$ be fixed.
 Let $x\in \X$, we have that
 \begin{align}
  \B_t^\discret(M)(x)
   & = \inf_{u \in \U} x^T C_t^\discret x + u^T D_t^\discret u + \lVert f_t^\discret\left(x,u \right)\rVert^2 \frac{f_t^\discret\left(x,u \right)^T}{\lVert f_t^\discret\left(x,u \right)\rVert} M \frac{f_t^\discret\left(x,u \right)}{\lVert f_t^\discret\left(x,u \right)\rVert} \notag
  \\
   & = \inf_{u \in \U} x^T C_t^\discret x + u^T D_t^\discret u + f_t^\discret\left(x,u \right)^T M f_t^\discret(x,u) \notag
  \\
   & = x^T C_t^\discret x + \inf_{u \in \U} u^T D_t^\discret u + f_t^\discret(x,u)^T M f_t^\discret(x,u). \label{equa:RandomRiccatiProof_determinist}
 \end{align}
 As $u\mapsto f_t^{\discret}(x,u)$ is linear, $D_t^\discret \succ 0$ and $M \succeq 0$, we have that
 \[
  g: u\in \U \mapsto u^T D_t^\discret u + f_t^\discret(x,u)^T Mf_t^\discret(x,u) \in \R
 \]
 is strictly convex, hence is minimal when $\nabla g(u) = 0$ \emph{i.e.} for $u\in \U$ such that:
 \begin{equation}
  \label{equa:optimalite_determ}
  \bp{D_t^\discret + \left( B_t^\discret \right)^T M B_t^\discret} u
  + \np{B_t^\discret}^T M \np{A_t^\discret}x = 0
  \eqfinp
 \end{equation}
 Now we will show that $D_t^\discret + \left( B_t^\discret \right)^T M B_t^\discret$ is invertible. As $M \in \Sy_n$ and $D_t^{\discret} \in \Sy_n^+$, for every $u \in \U$, we have:
 \begin{equation*}
  u^T \bp{D_t^\discret + \np{ B_t^\discret}^T M B_t^\discret} u
  = \underbrace{u^T D_t^\discret u}_{> 0} + \underbrace{\left( B_t^\discret  u\right)^T M \left(B_t^\discret u\right)}_{\geq 0 }> 0
  \eqfinp
 \end{equation*}
 We have shown that the symmetric matrix $D_t^\discret + \left( B_t^\discret \right)^T M B_t^\discret$ is definite positive and thus invertible. We conclude that Equation~\eqref{equa:optimalite_determ} is equivalent to:
 \begin{equation}\label{equa:inverseRiccati_determinist}
  u  = - \bp{D_t^\discret + \np{B_t^\discret}^T M B_t^\discret }^{-1} \left( B_t^\discret \right)^T M \left( A_t^\discret \right) x
  \eqfinp
 \end{equation}
 Finally, replacing Equation~\eqref{equa:inverseRiccati_determinist} in Equation~\eqref{equa:RandomRiccatiProof_determinist} we get after simplifications that
 \begin{align*}
  \B_t^\discret(M)(x)
   & = x^T \Big( C_t^\discret + \np{A_t^\discret}^T M A_t^\discret
  \\
   & \hspace{1cm} -\np{A_t^\discret}^T M B_t^\discret  \bp{D_t^\discret  + \np{B_t^\discret}^T MB_t^\discret}^{-1}
  \np{B_t^\discret}^{T} M A_t^\discret  \Big) x
  \eqfinv
 \end{align*}
 which gives Equation~\eqref{equa:Riccati_determinist}.

 \noindent $\bullet$ Equation~\eqref{equa:Riccati_reduced} follows from~\cite[Proposition 12.1.1 page 271]{La.Ro1995}.
\end{proof}

\section{Smallest and greatest eigenvalues}
\label{Appendix:Linear_Algebra}

Here we recall some formulas on the lowest and greatest eigenvalues of a matrix.
Denote the smallest eigenvalue of a symmetric real matrix $M$ by $\lambdamin(M)$
(every eigenvalue of $M$ is real) and by $\lambdamax(M)$ its greatest eigenvalue.

\begin{proposition}
 \label{prop:valeurpropre}
 Let $n>0$ be given. We have the following matrix inequalities.
 \begin{subequations}
  \begin{align}
    & \forall (A, B) \in \Sy_{n}^2
   \eqsepv  \lambdamax(A + B) \leq \lambdamax(A) + \lambdamax(B)
   \eqfinp
   \label{eq:vpmax-sum}
   \\
    & \forall (A,B) \in \Mat_n{\times}\Sy_{n}
   \eqsepv
   \lambdamax(A^TBA) \leq \lambdamax(A^TA) \lambdamax(B)
   \eqfinp
   \label{eq:vpmax-prod}
  \end{align}
 \end{subequations}
\end{proposition}

\newcommand{\spectral}[1]{{\lVert #1 \rVert}_{\text{sp}}}

\begin{proof}
 For any matrix $M \in \Mat_n$, the spectral norm of $M$, $\spectral{M}$, (See~\cite[Theorem 1.4.2]{Ci1989})
 is the subordinate matrix norm of the euclidean norm on $\R^n$. When the matrix $M\in \Sy_n$ is real symmetric, we have that $\spectral{M}= \lambdamax\np{M}$ and
 for any real matrix $M \in \Mat_n$, we have that $\lambdamax\np{M^T M} = \lambdamax\np{M M^T} = \lVert M \rVert^2$.

 \noindent $\bullet$ Fix $A,B \in \Sy_n$, we prove Equation~\eqref{eq:vpmax-sum}. As $A+B\in \Sy_n$ and using the fact that a subordinate matrix norm is a norm
 we have that
 \(
 \lambdamax\np{A + B} = \spectral{A + B}\leq \spectral{A}+\spectral{B} = \lambdamax(A) + \lambdamax(B)
 \).

 \noindent $\bullet$ Fix $\np{A,B} \in \Mat_n{\times}\Sy_{n}$. We prove Equation~\eqref{eq:vpmax-prod} as follows
 \begin{align*}
  \lambdamax\np{A^TBA}
   & = \lVert A^TBA \rVert
  \tag{as $A^TBA \in \Sy_n$}
  \\
   & \leq \spectral{A^T} \spectral{B}\spectral{A}
  \tag{$\spectral{\cdot}$ is submultiplicative as a matrix norm}
  \\
   & = \spectral{A}^2 \spectral{B}= \lambdamax(A^TA) \lambdamax(B)
  \eqfinp
 \end{align*}
 This ends the proof.
\end{proof}

\section{Homogenization}
\label{homogenization}
\newcommand{\Homogen}[1]{{\mathcal H}_{#1}}

We explain why, by adding another dimension to the state variable, there is no loss of generality induced by studying pure quadratic forms in Problem~\ref{pb:linear-quad-switch} instead of positive polynomial of degree $2$, nor is there one for studying linear dynamics instead of affine dynamics.

First, we define the operator $\Homogen{2}$ that maps a function $\func$ defined on a finite dimensional vector space $\ArbitrarySet$ to a $2$-homogeneous
function $\Homogen{2}\bp{\func}$ defined on the extended domain $\ArbitrarySet \times \R$ as
follows
\begin{equation}
 \label{h2def}
 \begin{array}{ccll}
  \Homogen{2} : & \overline{\R}^{\ArbitrarySet} & \longrightarrow & \overline{\R}^{\ArbitrarySet \times \R} \\
                & \func                         & \longmapsto     & \Homogen{2}\bp{\func}:
  \np{z,y} \mapsto y^2 \func \np{\frac{z}{y}} \ \text{if} \ y \neq 0, \ 0 \ \text{otherwise}.
 \end{array}
\end{equation}
Thus, if $\func$ is a positive polynomial of degree $2$, then $\Homogen{2}\np{\func}$ is a $2$-homogeneous convex quadratic form (with possibly a mixed term in $x$ and $u$).
In a similar way, we define the operator $\Homogen{1}$ that maps any function $\func$ defined on a domain $\ArbitrarySet$ and taking values in $\ArbitrarySet$
to a $1$-homogeneous function $\Homogen{1}\bp{\func}$ as follows
\begin{equation}
 \begin{array}{ccll}
  \Homogen{1} : & \ArbitrarySet^{\ArbitrarySet} & \longrightarrow & \np{\ArbitrarySet \times {\R}}^{\ArbitrarySet \times {\R}}                                                               \\
                & \func                         & \longmapsto     & \Homogen{1}\bp{\func}: \np{z,y} \mapsto \np{y \func \np{\frac{z}{y}} ,y} \ \text{if} \ y \neq 0, \ 0 \ \text{otherwise}.
 \end{array}
\end{equation}

Now consider $\left(\B_t\right)_{t\in \ce{0,T-1}}$ the Bellman operators associated to Problem~\ref{pb:linear-quad-switch}
\begin{equation}
 \begin{array}{ccll}
  \B_t: & \overline{\R}^{\X} & \longrightarrow & \overline{\R}^{\X}        \\
        & \func              & \longmapsto     & \B_t\bp{\func}: x \mapsto
  \min_{\substack{u \in \U                                                 \\\discret \in \Discret}} c_t^{\discret} \np{x,u} + \func \np{ f_t^{\discret}\np{x,u} }
  \eqfinv
 \end{array}
 \label{bellman-non-homogene}
\end{equation}

We denote by $\np{\B_t^{\Homogen{}}}_{t\in \ce{0,T-1}}$, the family of Bellman operators obtained through homogenization (with $\ArbitrarySet=\X\times\U$) as follows
\begin{equation}
 \begin{array}{ccll}
  \B_t^{\Homogen{}}: & \overline{\R}^{\X\times {\R}} & \longrightarrow & \overline{\R}^{\X \times {\R}}                                            \\
                     & \varphi                       & \longmapsto     & \B_t^{\Homogen{}} \bp{\varphi}: \np{x,y} \mapsto \min_{\substack{u \in \U \\\discret \in \Discret}} \Homogen{2}\bp{c_t^{\discret}}\bp{x,u,y}\\
                     &                               &                 & \hspace{3cm} + \varphi \np{ \Homogen{1} \bp{f_t^{\discret}}\np{x,u,y} } .
  \label{bellman-homogeneise}
 \end{array}
\end{equation}

The next proposition relates the solution of Problem~\ref{pb:linear-quad-switch} with non-homogeneous functions to the solution of the associated
homogenized problem.

\begin{proposition}
 \label{HomogeneVSnonhomogene}
 Let $\np{V_t}_{t\in \ce{0,T}}$ (resp. $\np{\widetilde{V}_t}_{t\in \ce{0,T}}$)
 denote the solutions of the Dynamic Programming Equation~\eqref{DynamicProgramming}
 system of equations associated with the operators $\np{\B_t}_{t\in \ce{0,T-1}}$ defined by Equation~\eqref{bellman-non-homogene}
 (resp. $\np{\B^{\Homogen{}}_t}_{t\in \ce{0,T-1}}$ defined by Equation~\eqref{bellman-homogeneise}) and final cost $\psi$
 (resp. $\Homogen{2}\bp{\psi}$)). Then, for every $x\in \X$ and $t\in \ce{0,T}$ , we have that
 \(
 V_t\np{x} = \widetilde{V}_t\np{x,1} \).
\end{proposition}

\begin{proof}
 First, it is easy to prove by backward recursion on time $t\in \ce{0,T}$, that the mappings $\widetilde{V}_t$ for every $t\in \ce{0,T}$, are $2$-homogeneous.
 Second, we will show by backward recursion on time that, for every $t\in \ce{0,T}$,
 \begin{equation}
  \label{eqtmp:1}
  \widetilde{V}_t =\Homogen{2}\bp{V_t}.
 \end{equation}
 Then, the result will follow by evaluating Equation~\eqref{eqtmp:1} at $y=1$.
 At the final time $t=T$, we have that
 \[
  \widetilde{V}_T := \Homogen{2}\bp{\psi} = \Homogen{2}\bp{V_T}.
 \]
 Now, assume that for some $t\in \ce{0,T-1}$, we have that $\widetilde{V}_{t+1} = \Homogen{2}\bp{V_{t+1}}$, for $(x,y)\in \X\times \R$ we successively obtain that
 \begin{align*}
  \widetilde{V}_t\np{x,y}
   & = \B_t^{\Homogen{}}\np{\widetilde{V}_{t+1}}\np{x,y}                 \\
   & = \min_{\substack{u \in \U                                          \\\discret \in \Discret}}\Homogen{2}\bp{c_t^{\discret}}\np{x,u,y} + \widetilde{V}_{t+1}\Bp{\Homogen{1}\bp{f_t^{\discret}}\np{x,u,y}}
  \tag{$2$-homogeneity of $\widetilde{V}_{t+1}$}                         \\
   & = \min_{\substack{u \in \U                                          \\\discret \in \Discret}}\Homogen{2}\bp{c_t^{\discret}}\np{x,u,y} +\Homogen{2}\bp{V_{t+1}} \Bp{yf_t^{\discret}\np{\frac{x}{y}, \frac{u}{y}}, y} \tag{Induction hyp. and def.}\\
   & = \min_{\substack{u \in \U                                          \\\discret \in \Discret}} y^2 c_t^{\discret}\np{\frac{x}{y}, \frac{u}{y}} +y^2 V_{t+1} \np{f_t^{\discret}\np{\frac{x}{y}, \frac{u}{y}}}
  \tag{by Equation~\eqref{h2def}}                                        \\
   & = y^2 \min_{\substack{u' \in \U                                     \\\discret \in \Discret}} c_t^{\discret}\np{\frac{x}{y}, u'} + V_{t+1}\np{f_t^{\discret}\np{\frac{x}{y},u'}} \tag{$u'= u/y$} \\
   & =y^2 \B_t\np{V_{t+1}}\np{\frac{x}{y}}                               \\
   & =\Homogen{2}\bp{V_t}\np{x,y} \tag{by Equation~\eqref{h2def}}\eqfinp
 \end{align*}
 This ends the proof.
\end{proof}

Lastly, we briefly explain how to get rid of the possible mixed terms in both $u$ and $x$ in the cost functions after homogenization. That is, there is no loss of generality to consider the case of cost functions which are positive polynomials of degree $2$ and affine cost than to consider the case studied in \S\ref{homogeneous_case}, \emph{i.e.} pure quadratic costs and linear functions. From Proposition~\ref{HomogeneVSnonhomogene}, we have seen that one can consider the case where the cost functions are $2$-homogeneous with linear dynamics. Assume (for the sake of simplicity, we omit the discrete control $v$ here) that the cost function $c_t$ is of the form
\[
 c_t\np{x,u} := x^T P_1 x + x^T P_2 u + u^T P_3 u,
\]
where $P_1$, $P_2$ and $P_3$ are symmetric semidefinite positive matrices of coherent dimensions, with $P_3$ being definite positive. Moreover, fix a $2$-homogeneous convex quadratic form $\psi$ and assume the dynamic $f_t$ to be linear of the form
\[
 f_t\np{x,u} := A x + Bu.
\]
Setting $Q_1 := P_1 - \frac{1}{4}P_2P_3^{-1}P_2^T$, $Q_2 := P_3$, $L := \frac{1}{2} P_3^{-1} P_2^{T}$, one has that the cost function $\np{x,u} \mapsto c_t'\np{x,u} := x^T Q_1 x + u^T Q_2 u$ is a quadratic function without mixing term and $\np{x,u} \mapsto f_t\np{x,u} := \np{A + L}x + Bu$ is linear. Furthermore, by straightforward computations, one can check that $c_t$ and $f_t$ satisfy:
\begin{equation}
 \label{eq:couts-dyn-sans-croises}
 c_t\np{x,u + Lx} = c_t'(x,u) \quad \text{and} \quad f_t\np{x,u+Lx} = f_t'(x,u).
\end{equation}
Note that as $Q_2 = P_3$, the matrix $Q_2$ is symmetric definite positive and as
$c_t$ is positive and by Equation~\eqref{eq:couts-dyn-sans-croises} for every $x\in \X$
and $u \in \U$
\[
 x^T Q_1 x = c_t'\np{x,0}=c_t\np{x,0 + Lx} \geq 0,
\]
then $Q_1$ is symmetric semidefinite positive. Thus the quadratic function $c_t'$ is convex and a pure quadratic form in the sense of Definition~\ref{purequadform}.

Consider the Bellman operator associated with the costs $c_t'$ and dynamics $f_t'$:
\begin{equation}
 \begin{array}{ccll}
  \B_t': & \overline{\R}^{\X} & \longrightarrow & \overline{\R}^{\X}         \\
         & \func              & \longmapsto     & \B_t'\bp{\func}: x \mapsto
  \min_{u \in \U}                                                 c_t' \np{x,u} + \func \np{ f_t'\np{x,u} }
  \eqfinv
 \end{array}
 \label{bellman-sans-terme-mixte}
\end{equation}
Thus, for any function $\func \in \overline{\R}^{\X}$ and every $x \in \X$, recall that $\U = \R^m$ is unconstrained, so we have that
\begin{align}
 \B_t\np{\func}\np{x} & = \min_{u \in \U} c_t(x,u) + \func\np{f_t\np{x,u}} \notag                       \\
                      & = \min_{u' = u + Lx \in \U} c_t\np{x, u + Lx} + \func\np{f_t\np{x,u+Lx}} \notag \\
                      & = \min_{u' \in \U} c_t'(x,u') + \func\np{f_t\np{x,u'}}    \notag                \\
 \B_t\np{\func}\np{x} & = \B_t'\np{\func}\np{x} \label{supersupersuper}.
\end{align}
From Equation~\eqref{supersupersuper}, one can deduce by backward recursion (as done in
Proposition~\ref{HomogeneVSnonhomogene}) on the time step $t\in \ce{0,T}$, that the value
functions $V_t$ (resp. $V_t'$) of the Dynamic Programming problem with Bellman
operators $\B_t$ (resp. $\B_t'$) and final cost function $\psi$ (resp. $\psi$ as
well) satisfy \( V_t = V_t'\) \eqfinp

\end{document}